\documentclass[reqno, 11pt]{amsart}

\usepackage{amsfonts,amsmath,amssymb,amsthm}
\usepackage{stmaryrd}

\usepackage[all]{xy}
\xyoption{matrix}
\usepackage[svgnames,hyperref]{xcolor}
\usepackage{extarrows}

\usepackage{changes}

\usepackage[english]{babel}
  
\usepackage{mathabx} 
\newcommand{\nicecolor}{Navy}
\usepackage[shortlabels]{enumitem}
\setlist[1]{wide}
\setlist[2]{leftmargin=15mm} 
\setlist[enumerate]{label=\rm{(\arabic*)}}
\setlist[enumerate,2]{label=\rm({\it\roman*}), }
\setlist[itemize]{label=\raisebox{0.25ex}{\tiny$\bullet$}}
\usepackage[backref, colorlinks, linktocpage, citecolor = \nicecolor, linkcolor = \nicecolor, linktoc=all]{hyperref} 
\usepackage{mathtools}

\newtheorem{lemma}{Lemma}[section]
\newtheorem{corollary}[lemma]{Corollary}

\newtheorem{proposition}[lemma]{Proposition}
\newtheorem{question}[lemma]{Question}

\newtheorem{maintheorem}{Theorem}

\newtheorem{mainprop}[maintheorem]{Proposition}

\theoremstyle{definition}
\newtheorem{definition}[lemma]{Definition}

\theoremstyle{remark}
\newtheorem{remark}[lemma]{Remark}

\newtheorem{example}[lemma]{Example}

\newcommand\C{{\mathbb C}}
\renewcommand\k{{\mathbf k}}
\newcommand\p{{\mathbb P}}

\newcommand\A{{\mathbb A}}
\newcommand\Z{{\mathbb Z}}

\newcommand\iso{\stackrel{\simeq}{\longrightarrow}}

\newcommand{\Aut}{\mathrm{Aut}}

\newcommand{\Spec}{\mathrm{Spec}}

\newcommand\CC{{\mathbb C}}

\newcommand\PP{{\mathbb P}}
\newcommand\FF{{\mathbb F}}
\newcommand\GG{{\mathbb G}}
\newcommand\NN{{\mathbb N}}
\newcommand\ZZ{{\mathbb Z}}
\newcommand\QQ{{\mathbb Q}}

\renewcommand\AA{{\mathbb A}}

\newcommand\rat{\dashrightarrow}

\newcommand\mm{{\mathfrak m}}

\DeclareMathOperator{\car}{char}
\DeclareMathOperator{\sing}{sing}
\DeclareMathOperator{\reg}{reg}
\DeclareMathOperator{\Pic}{Pic}
\DeclareMathOperator{\Var}{Var}
\newcommand{\et}{\textrm{\'et}}

\newcommand{\set}[2]{\{\,#1 \ | \ #2\,\}}

\newcommand{\kk}{\mathbf{k}}

\title[Complements of hypersurfaces]{Complements of hypersurfaces in projective spaces}
\author[J.~Blanc]{J\'er\'emy~Blanc}
\address{J\'er\'emy Blanc, Universit\"{a}t Basel, Departement Mathematik und Informatik, Spiegelgasse $1$, CH-$4051$ Basel, Switzerland.}
\email{jeremy.blanc@unibas.ch}
\urladdr{http://algebra.dmi.unibas.ch/blanc}

\author[P.-M.~Poloni]{Pierre-Marie~Poloni}
\address{Pierre-Marie~Poloni, Universit\"{a}t Basel, Departement Mathematik und Informatik, Spiegelgasse $1$, CH-$4051$ Basel, Switzerland.}
\email{ poloni.pierremarie@gmail.com}
\urladdr{http://algebra.dmi.unibas.ch/poloni}

\author[I.~Van~Santen]{Immanuel~Van~Santen}
\address{Immanuel Van Santen, Universit\"{a}t Basel, Departement Mathematik und Informatik, Spiegelgasse $1$, CH-$4051$ Basel, Switzerland.}
\email{immanuel.van.santen@unibas.ch}
\urladdr{http://algebra.dmi.unibas.ch/vansanten}

\thanks{The first author acknowledges support by the Swiss National Science Foundation Grant \textquotedblleft Geometrically ruled surfaces\textquotedblright 200020--192217.}


\begin{document}

\begin{abstract}
	We study the complement problem in projective spaces $\PP^n$
	over any algebraically closed field:
	If $H, H' \subseteq \PP^n$ are irreducible hypersurfaces of degree $d$ such that the complements
	$\PP^n \setminus H$, $\PP^n \setminus H'$ are isomorphic, are the hypersurfaces $H$, $H'$ isomorphic?
	
	For $n = 2$, the answer is positive if $d\leq 7$ and there are counterexamples when $d = 8$.
	In contrast we provide counterexamples for all $n, d \geq 3$ with $(n, d) \neq (3, 3)$.
	Moreover, we show that the complement problem has an affirmative answer for $d = 2$
	and give partial results in case $(n, d) = (3, 3)$. In the course of the exposition, we prove that rational normal projective surfaces 
	admitting a desingularisation by trees of smooth rational curves
	are piecewise isomorphic if and only if they coincide in the Grothendieck ring, answering affirmatively a question posed by Larsen and Lunts for such surfaces.
\end{abstract}

\maketitle 
\setcounter{tocdepth}{2}
\tableofcontents

\section{Introduction} 

Let $\k$ be a field, let $n\ge 2$ be an integer and let $f\in \k[x_0,\ldots,x_n]$ be an irreducible homogeneous polynomial of degree $d\ge 1$. The  variety 
\[
	\p^n_{f}=\set{[x_0:\cdots:x_n]\in \p^n}{f(x_0,\ldots,x_n)\neq0}
\] 
is the complement in $\p^n$ of the hypersurface given by $f=0$. It is an affine open subset of $\p^n$ with Picard group $\Z/d\Z$. Moreover, any  open subset of $\p^n$ isomorphic to it is of the form $\p^n_{g}$ for some  irreducible homogeneous polynomial $g\in \k[x_0,\ldots,x_n]$ of degree $d$ (see Lemma~\ref{Lem.Open_iso_to_Pn_g}).
One may then ask the following natural questions, for all 
irreducible homogeneous polynomials  $f,g\in \k[x_0,\ldots,x_n]$ of the same degree $d \geq 1$.

\begin{enumerate}[leftmargin=*]
\item \label{Quest1}
Does any isomorphism $\p^n_{f}\iso \p^n_{g}$ extend to an element of $\Aut(\p^n)$?
\item \label{Quest2}
If $\p^n_{f}$ and  $\p^n_{g}$ are isomorphic, is  there an element of $\Aut(\p^n)$ that sends $\p^n_{f}$ onto $\p^n_{g}$?
\item \label{Quest3}
If $\p^n_{f}$ and $\p^n_{g}$ are isomorphic, are the zero loci $V_{\p^n}(f)$ and $V_{\p^n}(g)$ of $f$ and $g$ isomorphic?
\end{enumerate}
Of course, any positive answer to one of the questions also gives a positive answer to the next one(s). 

If $d=1$, the first question has a negative answer, as there are many automorphisms of $\p^n_f\simeq \A^n$ that do not extend
to automorphisms of $\p^n$. However, the two other questions have a positive answer. 

For $d=2$, we can obtain similar results, in any dimension, when $\k$ is algebraically closed. Indeed, we will answer the first question in the negative   for any irreducible quadric, see~Example~\ref{Exa.Counterexample_Q1deg2}, 
and provide positive answers to the two other questions, see Theorem~\ref{Thm.Degree2} below.
 
The situation for $d=3$ is already more complicated, at least in dimension $n=3$.  
The first question has a negative answer for each singular cubic with Du Val singularities 
over an algebraically closed field of characteristic zero \cite[Theorem C and Theorem~4.3]{CheDub}, and 
more generally for each singular cubic 
over any algebraically closed field (Proposition~\ref{Prop:SingCubic}), 
but is wide open for any smooth cubic surface  \cite[Conjecture~1.1]{CheDub}. The two other questions are  open as well.

In dimension $n=2$, it was conjectured in \cite{Yoshihara} that the answer to the second question should be always positive
for algebraically closed fields of characteristic zero. This conjecture was  proved for all algebraically closed fields  if the curve $C\subseteq \p^2$ given by $f=0$ is such that $C\setminus L\simeq \A^1$ for some line (see \cite[Theorem]{Yoshihara} for a proof
in characteristic zero and  \cite[Theorem~1]{Hemmig} for a proof in any characteristic). The conjecture was however proven to be false in general, with a first example of degree $39$ {given in} \cite{BlancC}. A second family of counterexamples 
of degree $9+4m$ for each $m\ge 0$ was given {later} in \cite{Costa}. Finally, an example of degree $8$ and a proof that the conjecture is true for curves of degree $\le 7$ appeared in 
\cite[Corollary~1.2, Theorem~3]{Hemmig} for algebraically closed fields of any characteristic. Moreover, the examples given in  \cite{Hemmig} and \cite{BlancC} also give a negative answer to the third question.

In dimension $n\ge 3$, there was no negative answer to Questions~\ref{Quest2} and~\ref{Quest3} 
until now. In this article, we show by constructing explicit examples, 
that the answer of the third question, and thus of the two others, is negative for large degree and large dimension.
More precisely, in any dimension $n\ge 3$, the smallest degree for counterexamples is $3$, except maybe for $n=3$, where the smallest degree might be $4$. This contrasts the case of dimension $n=2$ where counterexamples 
for algebraically closed fields always have degree $\ge 8$.

\begin{maintheorem}\label{34and43}
	Let $\kk$ be a field and let $d,n \ge 3$ be integers such that $d$ is not a multiple of $\mathrm{char}(\kk)$ and such that $(d,n)\neq (3,3)$.
	
	There are two non-isomorphic irreducible hypersurfaces $H,H'\subseteq \p^n_\k$ of degree $d$, that have isomorphic complements.
\end{maintheorem}

We will prove Theorem~\ref{34and43} at the end of \S\ref{Subsec.Counter_examples}.  If $d = 2$, then Question~\ref{Quest2} has an affirmative answer, 
at least when $\k$ is algebraically closed.

\begin{maintheorem}
	\label{Thm.Degree2}
	Let $n \geq 2$ and let $\kk$  be an algebraically closed field.
	If the complements of two irreducible quadric hypersurfaces $H, H'$ in $\PP^n_{\kk}$ are isomorphic, then 
	there exists an automorphism of $\PP^n_\kk$ that sends $H$ onto $H'$.
\end{maintheorem}

 We will prove Thereom~\ref{Thm.Degree2} at the end of \S\ref{Subsec.Degree2}.
In the case where $n = 3$, we have the following partial positive result
concerning Question~\ref{Quest2}:

\begin{maintheorem}
	\label{Thm.Degree3n3}
	Let $\kk$ be an algebraically closed field and let $H, H' \subseteq \PP^3_{\kk}$
	be irreducible hypersurfaces such that their complements in $\PP^3_{\kk}$ are isomorphic.
	Then we have:
	\begin{enumerate}[leftmargin=*]
		\item \label{Thm.Degree3n3.1} 
		If both $H$ and $H'$ are normal, rational and each admits a desingularisation by trees of smooth rational curves, then 
		$H, H'$ are piecewise isomorphic, see Def.~$\ref{Def.piecewise_isomorphic}$;
		\item \label{Thm.Degree3n3.2} 
		If $H$ or $H'$ is a cubic, then $H, H'$ are piecewise isomorphic. 
		Moreover, $H$ is smooth if and only if $H'$ is smooth;
		\item \label{Thm.Degree3n3.3} 
		If $H$ or  $H'$ is a non-rational cubic, then there exists $\varphi \in \Aut(\PP^3_{\kk})$ 
		with $\varphi(H) = H'$.
	\end{enumerate}
\end{maintheorem}

We will prove Theorem~\ref{Thm.Degree3n3} at the end of \S\ref{Subsec.Degree3n3}. For $\kk = \CC$,
 it is proven in \cite{LaSe2012Birational-self-ma} that if $\PP^3_f \simeq \PP^3_g$ for irreducible polynomials $f, g$,
then $V_{\PP^2}(f)$ and $V_{\PP^2}(g)$ are piecewise isomorphic, which gives essentially Theorem~\ref{Thm.Degree3n3} over the field of complex numbers.

To prove Theorem~\ref{Thm.Degree3n3}, we study piecewise isomorphisms between surfaces. We prove in particular that every normal rational projective surface admitting a desingularisation by trees of smooth rational curves (e.g.~Du Val singularities) is piecewise isomorphic to the disjoint union of $\A^2$, one point and $n$ copies of $\A^1$ for some $n\ge 1$ (Proposition~\ref{prop:NormalTrees}). A similar result does not hold for all rational normal projective surfaces (Example~\ref{ExampleQuartics}). Using the Grothendieck ring $K_0(\Var_\k)$ and the topological Euler characteristic $($see \S\ref{sec.topEuler_characterstic_Grothendiecl_ring} for the definition$)$, one sees that the integer $n$ determines the piecewise isomorphism class. As a consequence of this, we give an affirmative answer to a question
posed by Larsen and Lunts (see \cite[Question 1.2]{LaLu2003Motivic-measures-a}) 
for rational projective normal surfaces admitting a desingularisation by trees of smooth rational curves
(see Proposition~\ref{Prop.Grothendieckring}).
This also partially generalises \cite[Theorem~4]{LiSe2010The-Grothendieck-r}.

\begin{mainprop}
	Let $\kk$ be algebraically closed and let $X, Y$ be rational 
	normal projective surfaces admitting a desingularisation by trees of smooth rational curves. 
	If the classes $[X]$ and $[Y]$ coincide in the Grothendieck ring
	$K_0(\Var_\k)$, then $X$ and $Y$ are piecewise isomorphic.
\end{mainprop}

\section{Lift to affine hypersurfaces}
Let $n\geq 1$. For each homogeneous polynomial $f\in \kk[x_0,\ldots,x_n]$ of degree $d\ge 1$ and each $\mu\in \k^*$, we denote by $X_{f,\mu}\subseteq \A^{n+1}$ the hypersurface given by 
\[
	X_{f,\mu}=\Spec(\k[x_0,\ldots,x_n]/(f-\mu)) \subseteq \AA^{n+1} \, , 
\] 
and obtain a canonical finite morphism
\[\begin{array}{cccc}
\pi_{f,\mu}\colon &X_{f, \mu}&\to& \p^n_f\\
& (x_0,\ldots,x_n)&\mapsto& [x_0:\cdots:x_n]
\end{array}\]
of degree $d$, which is an \'etale covering if $\car(\kk)$ does not divide $d$.

If $\kk$ is algebraically closed, all the varieties $X_{f,\mu}$ are isomorphic to $X_{f,1}$ by homotheties.

	If $\k$ is equal to the field of complex numbers $\CC$, then $\pi_{f,\mu}$
	is the universal abelian covering, as the abelianisation of the 
	fundamental group of $\p^n_f$  with respect to the Euclidean topology  is equal to $\ZZ / d \ZZ$ \cite[Proposition~2.3]{Li2005Lectures-on-topolo}. 
	
	We now prove that for each field $\kk$, the isomorphisms between complements of hypersurfaces lift to these coverings. This will be useful in the sequel, in order to study isomorphisms between varieties  $\p^n_f$ and  $\p^n_g$.

\begin{proposition}
\label{Prop.cyclic_covering}
Let $\kk$ be a field, let $n \geq 1$ be an integer and let $f,g\in \k[x_0,\ldots,x_n]$ be irreducible homogeneous polynomials 
of degree $d \ge 1$. Let $\Phi_0,\ldots,\Phi_n\in \k[x_0,\ldots,x_n]$ be homogeneous polynomials of degree $\ell\ge 1$.
Then, the following statements are equivalent:
\begin{enumerate}[leftmargin=*]
\item\label{cyclic_covering1} The rational map 
\[
	\begin{array}{rcl}
		\Phi\colon \p^n_f & \iso &  \p^n_g \, , \\
		{[x_0:\cdots:x_n]} & \mapsto & [\Phi_0(x_0,\ldots,x_n):\cdots: \Phi_n(x_0,\ldots,x_n)] \, ,
	\end{array}
\] 
is an isomorphism and the gcd of $\Phi_0, \ldots, \Phi_n$ is a power of $f$.
\item\label{cyclic_covering2}There exists $\mu\in \k^*$ such that the following map is an isomorphism
\[
	\begin{array}{rcl}
		\varphi\colon X_{f,1} & \iso & X_{g,\mu} \, , \\
		(x_0,\ldots,x_n) & \mapsto & (\Phi_0(x_0,\ldots,x_n),\ldots, \Phi_n(x_0,\ldots,x_n)) \, .
	\end{array}
\]  
\end{enumerate}
Moreover, if these statements hold, then $\ell$ is invertible  modulo $d$.
\end{proposition}
\begin{proof}
``\ref{cyclic_covering1} $\Rightarrow$ \ref{cyclic_covering2}'': 
By assumption there exists $r \geq 0$ and $\Phi_0', \ldots, \Phi_n'$ in $\kk[x_0, \ldots, x_n]$ 
without a common factor such that
$\Phi_i = f^r \Phi_i'$ for all $i$. The domain $\textrm{dom}(\Phi)$ of the rational self-map $\Phi \colon \PP^n \rat \PP^n$ 
is thus equal to 
$\PP^n \setminus V_{\PP^n}(\Phi'_0, \ldots, \Phi'_n)$ and contains $\PP^n_f$.
As $g$ is nowhere zero on $\p^n_g$, the polynomial $g(\Phi_0,\ldots,\Phi_n) = f^{dr} g(\Phi_0', \ldots, \Phi_n')$ is nowhere zero on $\p^n_f$ and is thus equal to $\mu f^t$ for some $\mu\in \k^*$ and some integer $t\ge 1$. We then obtain a morphism 
\[
	\varphi\colon X_{f,1} \to  X_{g,\mu} \, , \quad 
	(x_0,\ldots,x_n)\mapsto (\Phi_0(x_0,\ldots,x_n),\ldots, \Phi_n(x_0,\ldots,x_n)) \, .
\] 
It remains to see that $\varphi$ is an isomorphism.

The inverse of $\Phi$ is given by 
\[
	\Psi\colon \p^n_g\iso \p^n_f \, , \quad [x_0:\cdots:x_n]\mapsto [\Psi_0(x_0,\ldots,x_n):\cdots: \Psi_n(x_0,\ldots,x_n)] \, ,
\]
where $\Psi_0,\ldots,\Psi_n\in \k[x_0,\ldots,x_n]$ are homogeneous polynomials of the same degree $\ell'\ge 1$, without common factor. 

As $\Psi\circ \Phi=\mathrm{id}_{\p^n_f}$ and $\Phi\circ \Psi=\mathrm{id}_{\p^n_g}$, there exist homogeneous polynomials $A,B\in \k[x_0,\ldots,x_n]$ of degree $\ell\ell'-1$ such that 
\[\Psi_i(\Phi_0,\ldots,\Phi_n)= Ax_i \text{ and }\Phi_i(\Psi_0,\ldots,\Psi_n)= Bx_i \text{ for all }i\in \{0,\ldots,n\}.\]
Moreover, as the zero locus of $A$ is sent by $\Phi$
into the base-locus of $\Psi$
and the zero locus of $B$  is sent by $\Psi$ into the base-locus of 
$\Phi$, we find that $A$ is nowhere vanishing  on $\p^n_f$ and $B$ is nowhere vanishing on $\p^n_g$. We get  $A=\lambda f^s$ and $B=\lambda' g^{s}$ where $\lambda,\lambda'\in \k^*$ and where $s\ge 1$ is such that $ds=\ell\ell'-1$. In particular, $\ell$ is invertible modulo $d$. 
We  replace $\Psi_i$ with $\Psi_i/\lambda$ for each $i\in \{0,\ldots,n\}$ and reduce to the case where $\lambda=1$. 

Consider now the morphism
\[
	\psi\colon X_{g,\mu} \to \A^{n+1} \, , \quad 
	(x_0,\ldots,x_n)\mapsto (\Psi_0(x_0,\ldots,x_n),\ldots, \Psi_n(x_0,\ldots,x_n)) \, .
\] 
Using that  $\Psi_i(\Phi_0,\ldots,\Phi_n)= f^sx_i$ for each $i\in \{0,\ldots,n\}$, we obtain that 
$\psi\circ \varphi \colon X_{f, 1} \to \AA^{n+1}$ is the natural closed embedding. 
Hence, $(\psi \circ \varphi)^\ast = \varphi^\ast \circ \psi^\ast \colon \kk[\AA^{n+1}] \to \kk[X_{f, 1}]$  is surjective and $\varphi$ is dominant, i.e. $\varphi^\ast \colon \kk[X_{g, \mu}] \to \kk[X_{f, 1}]$ is injective. 
This implies that
$\varphi^\ast$ and thus $\varphi$ is an isomorphism.

``\ref{cyclic_covering2} $\Rightarrow$ \ref{cyclic_covering1}'': As $\Phi_0, \ldots, \Phi_{n}$ are homogeneous of degree $\ell$, we may define a rational self-map $\Phi\colon [x_0:\cdots:x_n]\mapsto [\Phi_0(x_0,\ldots,x_n):\cdots: \Phi_n(x_0,\ldots,x_n)]$ of $\p^n$. We now prove that $\Phi$ restricts to an isomorphism $\p^n_f\iso \p^n_g$ and that the gcd
of $\Phi_0, \ldots, \Phi_n$ is a power of $f$. 

By assumption $g(\Phi_0, \ldots, \Phi_n) = (f-1) p + \mu$ for some polynomial $p \in \kk[x_0, \ldots, x_n]$ 
of degree $r = d (\ell-1) \geq 0$.
Write $p = \sum_{i=0}^r p_i$ where $p_i$ is homogeneous of degree $i$.
Since $g(\Phi_0, \ldots, \Phi_n)$ is homogeneous, it follows that $p_0 = \mu$ and $p_{di} = f p_{d(i-1)}$ for all 
$1 \leq i \leq \ell-1$. This implies that
\[
	g(\Phi_0, \ldots, \Phi_n) = \mu f^{\ell} \, . 
\]
Since $g$ is homogeneous
and $f$ is irreducible  the gcd of $\Phi_0, \ldots, \Phi_n$ is a power of $f$.
Moreover, $\Phi$ restricts to a morphism $\Phi \colon \p^n_f\to \p^n_g$.

Using that $\pi_{g, \mu} \circ \varphi = \Phi \circ \pi_{f, 1}$, we get the following commutative diagram
\[
\xymatrix{
	\kk(X_{g, \mu})  \ar[r]^-{\varphi^\ast} & \kk(X_{f, 1}) \\
	\kk(\PP^n_g)  \ar[u]^-{\pi_{g, \mu}^\ast} \ar[r]^-{\Phi^\ast} & \kk(\PP^n_f) \ar[u]_-{\pi_{f, 1}^\ast}
}
\] 
and that $\Phi \colon \p^n_f\to \p^n_g$ is a quasi-finite surjection.
Note that $\pi_{f, 1}^\ast$ and $\pi_{g, \mu}^\ast$ are field extensions of degree $d$.
Since $\varphi^\ast$ is an isomorphism, $\Phi^\ast$ is an isomorphism as well, i.e.~$\Phi \colon \PP^n_f \to \PP^n_g$ 
is birational. Zariski's Main Theorem \cite[Corollaire~4.4.9]{Gr1961Elements-de-geomet-III} 
gives that $\Phi  \colon \PP^n_f \to \PP^n_g$ is an isomorphism.
\end{proof}

\begin{remark}\label{Rem.mu_equal_1}
	If there exists an isomorphism $\Phi \colon \PP^n_f \to \PP^n_g$, there are  homogeneous polynomials 
	$\Phi_0,\ldots,\Phi_n\in \k[x_0,\ldots,x_n]$ of degree $\ell \geq 1$ without common factor
	such that
	$\Phi([x_0: \cdots: x_n]) = [\Phi_0(x_0,\ldots,x_n):\cdots: \Phi_n(x_0,\ldots,x_n)]$. 
	Proposition~\ref{Prop.cyclic_covering} then gives the existence of $\mu\in \k^*$ such that 
	\[
	\begin{array}{rcl}
		\varphi\colon X_{f,1} & \iso & X_{g,\mu} \, , \\
		(x_0,\ldots,x_n) & \mapsto & (\Phi_0(x_0,\ldots,x_n),\ldots, \Phi_n(x_0,\ldots,x_n))
	\end{array}
\]   is an isomorphism. If $\kk$ is moreover algebraically closed, then
we may multiply each $\Phi_i$ with a root of $\mu^{-1}$ and assume that $\mu=1$. Hence, $X_{f,1}$ and $X_{g,1}$ are isomorphic.
\end{remark}

\begin{remark}\label{remark:Replacemodulo}
Let $\kk$ be a field,  $n \geq 1$ be an integer and let $f,g\in \k[x_0,\ldots,x_n]$ be irreducible homogeneous polynomials of degree $d \ge 1$. Let $e \ge 1$  be an integer and let $\Phi_0,\ldots,\Phi_n\in \k[x_0,\ldots,x_n]$ be polynomials such that the degree of each of their monomials is congruent to $e$ modulo $d$. If 
\[
	X_{f, 1} \iso X_{g, 1} \, , \quad (x_0,\ldots,x_n)\mapsto (\Phi_0(x_0,\ldots,x_n),\ldots, \Phi_n(x_0,\ldots,x_n))
\] 
is an isomorphism, then we may multiply the monomials of $\Phi_i$ with some powers of $f$ and assume that $\Phi_0,\ldots,\Phi_n$ are homogeneous of the same degree $\ell \geq 1$, and then apply Proposition~\ref{Prop.cyclic_covering} to obtain an isomorphism $\p^n_f\iso \p^n_g$.
\end{remark}

\begin{lemma}	\label{lemma:Replace_n_by_m}
Let $\kk$ be a field, $n \geq 1$ be an integer, $f,g\in \k[x_0,\ldots,x_n]$ be irreducible homogeneous polynomials of degree $d \ge 1$, and let $\Phi\colon \p^n_f  \iso   \p^n_g$ be an isomorphism given by homogeneous polynomials of degree $\ell$ such that their gcd is a power of $f$. 
If $\ell$ is congruent to $1$ modulo $d$, then $\p^{m}_f$ and $\p^{m}_g$ are isomorphic for each $m\ge n$.
\end{lemma}

\begin{proof}
We write $\Phi$ as $[x_0:\cdots:x_n] \mapsto  [\Phi_0(x_0,\ldots,x_n):\cdots: \Phi_n(x_0,\ldots,x_n)]$, where $\Phi_0,\ldots,\Phi_n\in \k[x_0,\ldots,x_n]$ are homogeneous of degree $\ell$ and the $\mathrm{gcd}$ of $\Phi_0,\ldots,\Phi_ n$ is a power of $f$. By Proposition~\ref{Prop.cyclic_covering}, the morphism
\[
\varphi\colon X_{f,1}  \iso  X_{g,\mu} \, , \quad (x_0,\ldots,x_n) \mapsto  (\Phi_0(x_0,\ldots,x_n),\ldots, \Phi_n(x_0,\ldots,x_n))
\]
is an isomorphism, where $X_{f,1}, X_{g,\mu}\subseteq \A^{n+1}$ are defined as before. If $\ell$ is congruent to $1$ modulo $d$, we can write $\ell=1+dr$ for some integer $r\ge 0$. For each $m\ge n$, 
we thus get an isomorphism
$X_{f,1}\times \A^{m-n}\iso X_{g,\mu}\times \A^{m-n}$  given by
\[
	(x_0,\ldots,x_m) \mapsto (\varphi(x_0,\ldots,x_m), x_{n+1}f(x_0,\ldots,x_n)^r,\ldots, x_m f(x_0,\ldots,x_n)^r)) \, .
\]
Applying again Proposition~\ref{Prop.cyclic_covering}, now in the converse direction, we get that the homogeneous polynomials $\Phi_0,\ldots,\Phi_n, x_{n+1} f^r,\ldots,x_m f^r$  of degree $\ell$ also define an isomorphism $\p^{m}_f\iso\p^{m}_g$. 
\end{proof}
The next examples shows that $\ell\neq 1$ in Proposition~\ref{Prop.cyclic_covering} 
	is possible.

\begin{example}
Let $\k$ be a field and let $f=g=xyz+x^3+y^3\in \k[x,y,z]$, which is the equation of a nodal cubic curve $\Gamma\subset\p^2$. By blowing-up the singular point, then blowing-up sucessively points on the exceptional divisor created lying on the curve, we obtain a birational morphism $X\to\p^2$ with a strict transform $\tilde{\Gamma}\subseteq X$
	of $\Gamma$ being a $(-1)$-curve. We then contract this curve and contract the exceptional divisors created, except the last one, to get an automorphism of $\p^2\setminus \Gamma$ that does not extend to $\p^2$ and has degree $8$. Explicitely, we define $\Phi_0,\Phi_1,\Phi_2\in \k[x,y,z]$, homogeneous of degree $8$, by
\begin{eqnarray*}
\Phi_0 & =& (-x^4z + 2x^3y^2 - 2x^2yz^2 + 2xy^3z + y^5 - y^2z^3)\cdot f \\
\Phi_1 & =& (x^2+yz)\cdot f^2 \\
\Phi_2 &=& 
x^7 y -x^6 z^2+6 x^5 y^2z-x^4 y^4-3 x^4 yz^3  +9 x^3 y^3z^2 \\
		&& + x^2y^5 z-3 x^2 y^2z^4 - x y^7+4 x y^4z^3 +2  y^6z^2- y^3z^5
\end{eqnarray*}
and  calculate that these define an involution 
\[
	\p^2_f\iso \p^2_f \, , \quad [x:y:z]\mapsto [\Phi_0(x,y,z):\Phi_1(x,y,z):\Phi_2(x,y,z)] \, .
\]
Here, the common degree is $8$, that is not congruent to $1$ modulo $3$, but invertible in $\ZZ/3\ZZ$, as Proposition~\ref{Prop.cyclic_covering} says.
\end{example}

\begin{lemma}
	\label{Lem.Open_iso_to_Pn_g}
	Let $\kk$ be a field, $n\ge 1$, $f \in \kk[x_0, \ldots, x_n]$ an irreducible homogeneous polynomial of degree $d \geq 1$
	and let $U \subseteq \p^n$ be an open subvariety.
	If there exists an isomorphism $\p^n_f \iso U$, then there exists an  irreducible homogeneous polynomial $g \in \kk[x_0, \ldots, x_n]$ of degree $d$ with $U = \p^n_g$.
\end{lemma}

Before we start the proof we recall the following fact: If $X$ is an affine irreducible 
normal variety over a field $\kk$ and if $U \subseteq X$ is a big open subset, i.e.~the codimension of $X \setminus U$ in $X$ is at least $2$,
then $\kk[U] = \kk[X]$ (see \cite[Theorem~11.5]{Ma1986Commutative-ring-t}). In particular, if $U$ is affine as well, then $U = X$.

\begin{proof}[Proof of Lemma~$\ref{Lem.Open_iso_to_Pn_g}$]
	If $U$ is big open in $\PP^n$, then for any open affine subvariety $V \subseteq \PP^n$,
	the affine subvariety $U \cap V$ is big open in $V$ and hence $U \cap V = V$. This would imply that $U = \PP^n$
	is affine, a contradiction since $n \geq 1$. Hence the union $Z$ of the $(n-1)$-dimensional
	irreducible components of $\PP^n \setminus U$ is non-empty. As $\PP^n \setminus Z$ is an affine
	big open subset of $U$ we get $\PP^n \setminus Z = U$.
	Moreover, $\kk[\p^n_f]^\ast= \k^\ast$ yields $\kk[U]^\ast=\k^\ast$, which implies that $Z$ is irreducible. 	
	Hence, there exists an irreducible homogeneous polynomial $g \in \kk[x_0, \ldots, x_n]$ such that $Z$
	is the zero locus of $g$, i.e.~$U = \PP^n_g$. Finally, $\ZZ / \deg(g) \ZZ \simeq \Pic(\p^n_g) \simeq \Pic(\p^n_f) \simeq \ZZ / d \ZZ$ implies that $\deg(g) = d$.
\end{proof}

\section{Non-isomorphic hypersurfaces  having isomorphic complements}
\subsection{Explicit isomorphisms between complements of projective cones}

The following formulas are 
inspired by  isomorphisms between cylinders over Danielewski surfaces given in \cite{MoPo2021Isomorphisms-betwe}.

\begin{lemma}\label{Lemma-Iso}
Let $\kk$ be a field, let $s,m,n\geq0$ be integers and let $S=\k[x_0,x_1,\ldots,x_s]$ be a polynomial ring in $s+1$ variables over $\kk$. 
Let $P,Q\in S[z]$ and denote by $H_P$ and $H_Q$ the hypersurfaces of $\A^{s+3}_{\k}=\Spec(S[y,z])$ defined by the equations 
\[
x_0^ny=P(x_0,\ldots,x_r,z) \quad \text{and} \quad x_0^ny=Q(x_0,\ldots,x_r,z) \, ,
\]
respectively.	Suppose that there exist $A,B\in S[z]$ such that
\begin{enumerate}[leftmargin=*]
\item $z-A(B(z))$ and $z-B(A(z))$ both belong to the ideal $x_0^{m}S[z]$,
\item $Q\left(A(z)+x_0^{m}w\right)\in x_0^nS[z,w]+P(z)S[z,w]$ and $P\left(B(z)+x_0^{m}w\right)\in x_0^nS[z,w]+Q(z)S[z,w]$, where $w$ is transcendental over $S[y,z]$.  
\end{enumerate}
Then, the following are inverse isomorphisms.
\[\begin{array}{rcl}
	\Phi\colon H_P\times\A^1_{\k} &\iso&  H_Q\times\A^1_{\k}\\
	(x_0,\ldots,x_s,y,z,w) &\mapsto& (x_0,\ldots,x_s,\frac{Q\left(A(z)+x_0^{m}w\right)}{x_0^n}, \\
	&& A(z)+x_0^{m}w,\frac{z-B\left(A(z)+x_0^{m}w\right)}{x_0^{m}}) \, ,\\
	\Psi\colon  H_Q \times\A^1_{\k} &\iso&  H_P\times\A^1_{\k}\\
	(x_0,\ldots,x_s,y,z,w) &\mapsto& (x_0,\ldots,x_s,\frac{P\left(B(z)+x_0^{m}w\right)}{x_0^n}, \\
	&& B(z)+x_0^{m}w,\frac{z-A\left(B(z)+x_0^{m}w\right)}{x_0^{m}}).
\end{array}\]
\end{lemma}

\begin{proof}
The hypotheses on $A$ and $B$ imply that $\Phi$ and $\Psi$ are morphisms. Moreover, the compositions are given by  
\[(\Psi\circ\Phi)(x_0,\ldots,x_s,y,z,w)=(x_0,\ldots,x_s,P(z)/x_0^n,z,w),\]
\[(\Phi\circ\Psi)(x_0,\ldots,x_s,y,z,w)=(x_0,\ldots,x_s,Q(z)/x_0^n,z,w).\]
\end{proof}

\begin{proposition}\label{Prop.Counterexample.dim3}Let $\kk$ be a field, $d\ge 4$ be integers and let $f,g\in\k[x,y,z]$  be the homogeneous polynomials  of degree $d$ defined by \[f=x^{d-1}y+z^{d} \text{ and } g=x^{d-1}y+z^{d}+dx^{d-2}z^2.\] Then, for each $r\ge 1$, the 
	open subvarieties  $\p^{r+2}_f,\p^{r+2}_g \subseteq \p^{r+2}$ are isomorphic.
\end{proposition}

\begin{remark}
The open subvarieties $\p^{2}_f,\p^{2}_g \subseteq \p^{2}$ are not isomorphic  when $\kk=\mathbb{C}$ is the field of complex numbers. This follows from Proposition~\ref{Prop.cyclic_covering} and from the fact that, for every $\mu\in\C^*$,  the  (Danielewski) hypersurfaces $X_{f,1}, X_{g,\mu}\subseteq\A^3_{\mathbb{C}}$ defined by the equations $f=1$ and $g=\mu$, respectively, are  not isomorphic  (see \cite[Theorem 9]{PoloniClassifications}).
\end{remark}

\begin{proof}
We define $s=0$,  $x_0=x$, $S=\kk[x]$, $m=d$, $n=d-1$, 
\[
	\begin{array}{rclrcl}
		P(z)&=&1-z^{d}, & Q(z) &=&1-z^{d}-dx^{d-2}z^2 \, , \\
 		A(z)&=&z-x^{d-2}z^3, & B(z) &=&z+x^{d-2}z^3 \, ,
 	\end{array}
 \]
and prove that the hypotheses of Lemma~\ref{Lemma-Iso} are satisfied for the above polynomials in $S[z]$.

Firstly, we have $A(z)\equiv z\equiv B(z) \pmod{x^{d-2}}$. Thus, since $d\le 2(d-2)$, as $d\ge 4$, we get $x^{d-2}A(z)^3\equiv x^{d-2}z^3\equiv x^{d-2} B(z)^3 \pmod{x^d}$ and we  obtain
\begin{eqnarray*}
	z-A(B(z))&=& z-B(z)+x^{d-2}B(z)^3\equiv 0\pmod{x^d} \, , \\
	z-B(A(z))&=& z-A(z)-x^{d-2}A(z)^3\equiv 0\pmod{x^d} \, .
\end{eqnarray*}

We then check that $P\left(B(z)+x^{d}w\right)\in x^{d-1}S[z,w]+Q(z)S[z,w]$:
\begin{align*}
P(B(z)+x^{d}w) &\equiv P(B(z)) \\
&\equiv 1 -(z+x^{d-2}z^3)^{d} \\
&\equiv 1-z^{d}-dx^{d-2}z^{d+2}\\
&\equiv (1-z^{d}-dx^{d-2}z^2)\cdot (1+dx^{d-2}z^2)\\
&\equiv Q\cdot (1+dx^{d-2}z^2)\, \pmod{x^{d-1}}.
\end{align*}
Similarly, we have
$Q(A(z)+x^{d}w)\in x^{d-1}S[z,w]+P(z)S[z,w]$:
\begin{align*}
Q(A(z)+x^{d}w) &\equiv Q\left(A(z)\right)\\
& \equiv 1-(z-x^{d-2}z^3)^{d}-dx^{d-2}(z-x^{d-2}z^3)^2 \\
&\equiv 1 -z^{d}+dx^{d-2}z^{d+2}-dx^{d-2}z^2\\
&\equiv (1-z^{d})\cdot (1-dx^{d-2}z^2)\\
&\equiv P\cdot (1-dx^{d-2}z^2)\, \pmod{x^{d-1}}.
\end{align*}
Noting that $f-1=x^{d-1}y-P(z)$ and $g-1=x^{d-1}y-Q(z)$, we may now apply Lemma~\ref{Lemma-Iso} to obtain an isomorphism
\[ \begin{array}{rcl}
	\varphi\colon\Spec(\k[x,y,z,w]/(f-1)) & \iso &  \Spec(\k[x,y,z,w]/(g-1)) \\
	(x,y,z,w) & \mapsto &  (x,\frac{Q(A(z)+x^{d}w)}{x^{d-1}}, 
			A(z)+x^{d}w, \\
			&& \frac{z-B(A(z)+x^{d}w)}{x^{d}}).
	\end{array}
\]

Arguing as in  Remark~\ref{remark:Replacemodulo}, we may then use this isomorphism to construct an isomorphism $\p^3_f\iso \p^3_g$. Indeed, the second component of $\varphi$  can be replaced by a polynomial, since 
\[
	Q\left(A(z)+x^{d}w\right)\equiv P\cdot  (1-dx^{d-2}z^2)\pmod{x^{d-1}}
\] 
and  $P\equiv x^{d-1}y\pmod{f-1}$.
Moreover, the  fourth component of $\varphi$ is a polynomial, since $z-A(B(z)) \equiv 0\pmod{x^d}$.
Doing this, we obtain an expression for $\varphi$ given by four polynomials. It remains to check that each monomial appearing in that polynomial expression of $\varphi$ has degree congruent to $1$ modulo $d$. Working modulo $d$ with the degrees, we observe that $A,B,P,Q$ are homogeneous of degree $1,1,0,0$, respectively, hence that the polynomials $z-B(A(z)+x^{d}w)$, $Q(A(z)+x^dw)$ and $P-x^{d-1}y$ are homogeneous of degree $1$, $0$ and $0$, respectively. Therefore, we can conclude that $\p^3_f\iso \p^3_g$ as in Remark~\ref{remark:Replacemodulo} and  furthermore that $\p^{r+2}_f\iso \p^{r+2}_g$ for all $r\geq 1$ by Lemma~\ref{lemma:Replace_n_by_m}.
\end{proof}

\begin{proposition}
	\label{Prop.Counterexample.dim4}
	Let $\kk$ be a field, $d\ge 3$ be an integer and let $f,g\in\k[x_0,x_1,y,z]$  be the homogeneous polynomials  of degree $d$ defined by 
	\[
	f=x_0^{d-1}y+z^{d} \quad \textrm{and} \quad
	g=x_0^{d-1}y+z^{d}+dx_0^{d-2}x_1^2 \, . \] 
	Then, for each $r\ge 1$,  the open subvarieties $\p^{r+3}_f,\p^{r+3}_g \subseteq \p^{r+3}$ are isomorphic. 
\end{proposition}

\begin{remark}
	We don't know, whether the open subvarieties $\p^{3}_f,\p^{3}_g \subseteq \p^{3}$ are isomorphic, 
	or even if the hypersurfaces $X_{f,1}, X_{g,1}\subseteq \A^3$ given by $f=1$ and $g=1$, respectively, are isomorphic or not.
\end{remark}

\begin{proof}[Proof of Proposition~$\ref{Prop.Counterexample.dim4}$]
We define $s=1$, $S=\kk[x_0,x_1]$, $m=d$, $n=d-1$, $\Delta=x_0^{d-2}x_1^2$, 
\[
	\begin{array}{rclrcl}
	P(z) &=&1-z^{d}, & Q(z) &=&1-z^{d}-d\Delta,\\
 	A(z) &=&z(1-\Delta+\Delta^2), & B(z) &=&z(1+\Delta) \, ,
	\end{array}
\]
and prove that the hypotheses of Lemma~\ref{Lemma-Iso} are satisfied for the above polynomials in $S[z]$.

Firstly, we calculate $A(B(z))=B(A(z))=z(1-\Delta+\Delta^2)(1+\Delta)=z(1+\Delta^3)$. As $d\ge 3$, $\Delta^3$ is divisible by $x_0^{3d-6}$. Hence, we find that $A(B(z))=B(A(z))$ is congruent to $z$  modulo $x_0^d$.

We then check that $P(B(z)+x_0^{d}w)\in x_0^{d-1}S[z,w]+Q(z)S[z,w]$. As $d\ge 3$,  $\Delta^2$ is divisible by $x_0^{d-1}$. Hence:
\begin{align*}
P(B(z)+x_0^{d}w) &\equiv P(B(z))\\
 &\equiv 1-z^d(1+\Delta)^{d} \\
&\equiv 1- z^d(1+d\Delta) \\
&\equiv(1-z^{d}-d\Delta)\cdot (1+d\Delta)\\
&\equiv Q\cdot (1+d\Delta)\, \pmod{x_0^{d-1}}.
\end{align*}

Similarly, we have
$Q(A(z)+x_0^{d}w)\in x_0^{d-1}S[z,w]+P(z)S[z,w]$:
\begin{align*}
Q(A(z)+x_0^{d}w) &\equiv Q(A(z))\\
& \equiv  1-z^d(1-\Delta)^d-d\Delta\\
& \equiv  1-z^d(1-d\Delta)-d\Delta\\
& \equiv  (1-z^{d})\cdot (1-d\Delta)\\
&\equiv P\cdot (1-d\Delta)\, \pmod{x_0^{d-1}}.
\end{align*}

Since $f-1=x_0^{d-1}y-P(z)$ and $g-1=x_0^{d-1}y-Q(z)$, we may now apply Lemma~\ref{Lemma-Iso} to obtain an isomorphism
\[
\begin{array}{rcl}	
	\varphi\colon\Spec(\k[x_0,x_1,y,z,w]/(f-1)) & \iso &  \Spec(\k[x_0,x_1,y,z,w]/(g-1)) \\
	(x_0,x_1,y,z,w) & \mapsto & (x_0,x_1,\frac{Q(A(z)+x_0^{d}w)}{x_0^{d-1}},A(z)+x_0^{d}w, \\
							 && \frac{z-B(A(z)+x_0^{d}w)}{x_0^{d}})\, .
\end{array}
\]

We again proceed as in Remark~\ref{remark:Replacemodulo}.
The third component of $\varphi$ can be expressed by a polynomial, since 
\[
	Q(A(z)+x_0^{d}w)\equiv P\cdot  (1- d x_0^{d-2}x_1^2)\pmod{x_0^{d-1}}
\]
and  $P\equiv x_0^{d-1}y\pmod{f-1}$.
The last component of $\varphi$ is already a polynomial, as its numerator is a multiple of its denominator. It remains to check that  each monomial appearing in the expression of $\varphi$ has degree congruent to $1$ modulo $d$. For this, we simply work with the degrees modulo $d$ and observe that $\Delta$ is homogeneous of degree $0$, so that $A,B,P,Q$ are homogeneous of degree $1,1,0,0$, respectively. Hence, $z-B(A(z)+x_0^{d}w)$, $Q(A(z)+x_0^dw)$ and $P-x_0^{d-1}y$ are homogeneous of degree $1$, $0$ and $0$, respectively. Therefore, we can conclude by  Remark~\ref{remark:Replacemodulo} and  Lemma~\ref{lemma:Replace_n_by_m}.
\end{proof}

\subsection{Non-isomorphic hypersurfaces}
\label{Subsec.Counter_examples}
We now want to prove that the hypersurfaces of Propositions~\ref{Prop.Counterexample.dim3} and \ref{Prop.Counterexample.dim4} are not isomorphic (except when the characteristic divides $d$, in which case both are in fact equal).

Recall that the multiplicity  $\mathrm{mult}_p(X)$ of a variety $X$ at a point $p$ is the multiplicity of the maximal ideal $\mathfrak{m}_{X,p}$ in the local ring $\mathcal{O}_{X,p}$. 
\begin{lemma}\label{lemm:MultMaxLin}
Let $n\ge 1$ be an integer, let $H\subseteq \p^n$ be a hypersurface of degree $d\ge 1$. The set $M=\{p\in H\mid \mathrm{mult}_p(H)=d\}$ is either empty or a linear subspace of $\p^n$. 
\end{lemma}
\begin{proof}
It suffices to take two different points $p,q\in M$ and to prove that the line containing $p$ and $q$ is contained in $M$. Changing coordinates, we may assume that $p=[0:\cdots:0:1]$ and $q=[0:\cdots:0:1:0]$. The equation of $H$ is then given by a non-zero homogeneous polynomial $P\in \k[x_0,\ldots,x_{n-2}]$ of degree $d \geq 1$ (and thus we have $n\ge 2$). Hence the line $V_{\PP^n}(x_0, \ldots, x_{n-2})$ (which passes through $p$ and $q$) is contained in $M$. 
\end{proof}
\begin{lemma}\label{Lem:IsoHypNNplus}
Let $\kk$ be a field, let $n\ge 1$ and let $f,g\in \kk[x_0,\ldots,x_n]$ be homogeneous polynomials of the same degree $d\ge 1$. If $V_{\p^{n+1}}(f)$ and $V_{\p^{n+1}}(g)$ are isomorphic, then $V_{\p^{n}}(f)$ and $V_{\p^{n}}(g)$ are isomorphic.
\end{lemma}
\begin{proof}
We write $H= V_{\p^{n+1}}(f)$, $H'= V_{\p^{n+1}}(g)$. 
Let $\varphi\colon H\iso H'$ be an isomorphism. It then sends $M$ onto $M'$, where $M=\set{x\in H}{\mathrm{mult}_p(H)=d}$ and $M'=\set{x\in H'}{\mathrm{mult}_p(H')=d}$.  As $f,g\in \k[x_0,\ldots,x_n]$, the point $q=[0:\cdots:0:1]\in \PP^{n+1}$ belongs to both $M$ and $M'$. By Lemma~\ref{lemm:MultMaxLin}, $M$ and $M'$ are linear subspaces of $\p^{n+1}$ given by $r\ge 1$ 
linearly independent linear polynomials in the variables $x_0,\ldots,x_n$ (and $r \leq n+1$). Applying a change of coordinates on $x_0,\ldots,x_n$, we may assume that $M=M'=V_{\p^{n+1}}(x_0,\ldots,x_{r-1})$. The restriction of $\varphi$ then becomes a linear automorphism of $M=M'$, and $f,g\in \k[x_0,\ldots,x_{r-1}]$. Replacing $\varphi$ with $\varphi\circ \alpha$, where $\alpha\in \Aut(\p^{n+1})$ does nothing on $x_0,\ldots,x_{r-1}$ and is a linear automorphism on $x_r,\ldots,x_{n+1}$, we may assume that $\varphi$ fixes the point $q=[0:0:\cdots:0:1]$.  It then induces an isomorphism $\hat{H}\iso \hat{H}'$, where $\hat{H}$, $\hat{H}'$ are the blow-ups of $H$ and $H'$ at the point $q$. The exceptional divisors are isomorphic to  $V_{\p^{n}}(f)$ and  $V_{\p^{n}}(g)$,  respectively, so we find that these hypersurfaces are isomorphic.
\end{proof}
\begin{lemma}\label{LemmCurveD}
Let $\kk$ be a field, let $d\ge 4$ be an integer that is not a multiple of $\mathrm{char}(\kk)$ and let 
\[f=x^{d-1}y+z^{d} \text{ and } g=x^{d-1}y+z^{d}+\mu x^{d-2}z^2,\] 
where $\mu\in \k$. Suppose that $\alpha\colon V_{\p^2}(f)\iso V_{\p^2}(g)$ is an isomorphism. Then, $\mu=0$ and $\alpha$ is equal to $[x:y:z]\mapsto [a^d x:y:a^{d-1} z]$ for some $a\in \k^*$.
\end{lemma}
\begin{proof}
Expressing $y$ in terms of $x,z$, we obtain two birational morphisms
\[\begin{array}{rrcl}
\tau\colon & \p^1 & \to&  V_{\p^2}(f)\\
 & {[u:v]} & \mapsto&  [u^d:-v^d:u^{d-1}v]
 \end{array}\]
and
\[\begin{array}{rrcl}
\tau'\colon & \p^1 & \to&  V_{\p^2}(g)\\
& {[u:v]} & \mapsto&  [u^d:-v^d-\mu u^{d-2}v^2:u^{d-1}v]
\end{array}
\]
whose inverses are given by $[x:y:z]\mapsto [x:z]$. They induce isomorphisms $\p^1\setminus \{[0:1]\}\iso V_{\p^2}(f)\setminus \{[0:1:0]\}$ and $\p^1\setminus \{[0:1]\}\iso V_{\p^2}(g)\setminus \{[0:1:0]\}$. 

Both $\tau$ and $\tau'$ are bijective and they send $[0:1]$ onto $[0:1:0]$, which is the unique singular point of $V_{\p^{2}}(f)$ and $V_{\p^{2}}(g)$,
respectively. Since the isomorphism $\alpha\colon V_{\p^{2}}(f)\iso V_{\p^{2}}(g)$ must fix the point $[0:1:0]$,  the birational map $\hat\alpha=(\tau')^{-1}\alpha\tau$ is an automorphism of $\p^1$ that fixes the point $[0:1]$.  Thus, $\hat\alpha$ is  of the form $[u:v]\mapsto [u:av+bu]$ for some $a\in \k^*$, $b\in \k$.

Note that $s_a\colon [x:y:z]\mapsto [a^d x:y:a^{d-1} z]$ is an automorphism of $V_{\p^2}(f)$, that lifts to $\hat s_a=(\tau)^{-1}s_a\tau=[u:v]\mapsto [a u:v]$. Replacing $\alpha$ with $\alpha s_a$, we replace $\hat\alpha$ with $\hat\alpha \hat s_a$, 
and we may assume that $a=1$. It remains to see that $b=\mu=0$.
We calculate 
\begin{eqnarray*}
		\alpha( [x:y:z]) &=& \tau' \hat\alpha \tau^{-1}( [x:y:z])\\
								 &=& \tau' \hat\alpha([x:z]) \\
								 &=& \tau' ([x:z+bx]) \\
								 &=& [x^d:-(z+bx)^d-\mu x^{d-2}(z+bx)^2:x^{d-1}(z+bx)] \, .
\end{eqnarray*}

As $\alpha$ fixes $[0:1:0]$, the pull-back by $\alpha$ of $x/y$ and $z/y$ are rational functions 
\[
	\frac{x^d}{-(z+bx)^d-\mu x^{d-2}(z+bx)^2} \quad \text{and} \quad
	\frac{x^{d-1}}{-(z+bx)^{d-1}-\mu x^{d-2}(z+bx)} \, ,
\] 
whose restrictions to $V_{\p^2}(f)$ are regular at $[0:1:0]$. For each $j\in \{d-1,d\}$, there are thus two homogeneous polynomials $P_j,Q_j\in\k[x,y,z]$, of the same degree, such that $Q_j(0,1,0)\neq 0$ and such that \[\frac{P_j(x,y,z)}{Q_j(x,y,z)}=\frac{x^{j}}{-(z+bx)^{j}-\mu x^{d-2}(z+bx)^{j+2-d}}\] on the curve $V_{\p^2}(f)$. Using the  morphism 
\[
	\A^1\to V_{\p^2}(f) \, ,  \quad t\mapsto \tau([t:1])=[t^d:-1:t^{d-1}]=[t:-t^{1-d}:1] \, ,
\]
 we find
\[
	\frac{P_j(t^d,-1,t^{d-1})}{Q_j(t^d,-1,t^{d-1})}=\frac{t^{j}}{B_j} \, ,
\]
where $B_j = -(1+bt)^{j}-\mu t^{d-2}(1+bt)^{j+2-d} \in \kk[t]$.
As $Q_j(0,1,0)\neq 0$, the polynomial $Q_j(t^d,-1,t^{d-1})\in \k[t]$ is not divisible by~$t$. As $B_j(0) \neq 0$ 
(we use here $d \geq 4$), $t^j$ divides $P_j(t^d,-1,t^{d-1})$. There is thus $A_j\in \k[t]$ such that 
\[
	P_j(t^d,-1,t^{d-1})=A_j\cdot t^j \quad \text{and} \quad Q_j(t^d,-1,t^{d-1})=A_j\cdot B_j \, .
\]

We consider now the case $j=d$ and prove that $b=0$. We shall afterwards prove $\mu=0$  by considering the case $j=d-1$.

As $d\ge 4$, the equality $P_d(t^d,-1,t^{d-1})=A_d\cdot t^d$ gives the existence of $\epsilon\in \k$ such that $A_d\equiv \epsilon \pmod{t^{2}}$. We moreover have $B_d\equiv -1-bdt \pmod{t^2}$. This gives 
\[Q_d(t^d,-1,t^{d-1})=A_d\cdot B_d\equiv -\epsilon -\epsilon bdt \pmod{t^2}.\]
The polynomial $Q(t^d,-1,t^{d-1})$ is not divisible by $t$, so $\epsilon\neq 0$, and its coefficient of $t$ is zero, as $d\ge 4$. Since $\mathrm{char}(\k)$ does not divide $d$, this gives $b=0$. 

We now use $j=d-1$. As $d\ge 4$, the equality $P_{d-1}(t^d,-1,t^{d-1})=A_{d-1}\cdot t^{d-1}$ gives the existence of $\xi,\xi'\in \k$ such that $A_{d-1}\equiv \xi+t \xi' \pmod{t^{d-1}}$. As $b=0$, we find $B_{d-1}=B_d=-1-\mu t^{d-2}$. This gives 
\[Q_{d-1}(t^d,-1,t^{d-1})=A_{d-1}\cdot B_{d-1}\equiv -\xi-t \xi'-\mu\xi t^{d-2} \pmod{t^{d-1}}.\]
As $t$ does not divide this polynomial, $\xi\neq0$. Moreover, we obtain $\xi'=0$ and $\mu=0$, as $d\ge 4$.
\end{proof}

\begin{proposition}\label{Prop.IsoHypNN3}
Let $\kk$ be a field, let $d\ge 4$ be an integer that is not a multiple of $\mathrm{char}(\kk)$ and  let $f,g\in\kk[x,y,z]$ be defined by \[f=x^{d-1}y+z^{d} \text{ and } g=x^{d-1}y+z^{d}+\mu x^{d-2}z^2,\] where $\mu\in \k^*$. Then, for each $r\ge 0$, the two hypersurfaces  $V_{\p^{r+2}}(f)$, $V_{\p^{r+2}}(g)$  in $\p^{r+2}$   are not isomorphic.
\end{proposition}

\begin{proof}
By Lemma~\ref{LemmCurveD}, the result is true when $r=0$. Using Lemma~\ref{Lem:IsoHypNNplus}, we can then argue by induction and obtain the result for every integer $r$.
\end{proof}

\begin{proposition}
	\label{Prop.IsoHypNN4}
	Let $\kk$ be a field, $d \geq 3$ be an integer and  let $f,g\in\kk[x_0,x_1,y,z]$ be defined by
	\[
		f=x_0^{d-1}y+z^{d} \, , \quad g=x_0^{d-1}y+z^{d}+\mu x_0^{d-2}x_1^2 \, ,
	\] where $\mu\in \k^*$. Then, for each $r\ge 0$, the two hypersurfaces  $V_{\p^{r+3}}(f)$, $V_{\p^{r+3}}(g)$  in $\p^{r+3}$   are not isomorphic.
\end{proposition}

\begin{proof}
	By Lemma~\ref{Lem:IsoHypNNplus},  it suffices to consider the case $r=0$ and to prove that $V_{\p^{3}}(f)$ and $V_{\p^{3}}(g)$ are not isomorphic. Looking at the derivative with respect to $y$, the singular locus of both hypersurfaces is contained in the line $\ell\subseteq \p^3$ 
	given by $x_0 = z = 0$. 
	
	The surface $V_{\p^{3}}(f)$ has multiplicity $d$ at the point where $x_0=y=z=0$. 
	
	It remains to see that $V_{\p^{3}}(g)$ has multiplicity $< d$ at every point. For this, write $g=z^d+ x_0^{d-2}(x_0y+\mu x_1^2)$ and observe that $V_{\PP^3}(x_0y+\mu x_1^2)$ is smooth outside $x_0=x_1=y=0$ and thus on $\ell$.
\end{proof}

We may now prove Theorem~\ref{34and43}, which is a direct consequence of Propositions~\ref{Prop.Counterexample.dim3}, \ref{Prop.Counterexample.dim4}, \ref{Prop.IsoHypNN3} and \ref{Prop.IsoHypNN4}.

\begin{proof}[Proof of Theorem~\ref{34and43}]
As in the statement, we take a field $\kk$, and integers $d,n\ge 3$ such that $d$ is not a multiple of $\mathrm{char}(\kk)$ and such that $(d,n)\neq (3,3)$. We are looking for hypersurfaces $H=V_{\p^n}(f),H'=V_{\p^n}(g)\subseteq \p^n_\k$ that are not isomorphic but have isomorphic complements.

If $d\ge 4$, we may choose $f=x^{d-1}y+z^{d}$ and $g=x^{d-1}y+z^{d}+dx^{d-2}z^2$, that we see in $\k[x,y,z,w_1,\ldots,w_r]$ with $r=n-2\ge 1$. By Propositions~\ref{Prop.Counterexample.dim3}, the complements of $H$ and $H'$ are isomorphic, and by Proposition~\ref{Prop.IsoHypNN3}, the hypersurfaces $H$ and $H'$ are not isomorphic.

If $n\ge 4$, we may choose $f=x_0^{d-1}y+z^{d}$ and $g=x_0^{d-1}y+z^{d}+dx_0^{d-2}x_1^2$, that we see in $\k[x_0,x_1,y,z,w_1,\ldots,w_r]$ with $r=n-3\ge 1$. By Propositions~\ref{Prop.Counterexample.dim4}, the complements of $H$ and $H'$ are isomorphic, and by Proposition~\ref{Prop.IsoHypNN4}, the hypersurfaces $H$ and $H'$ are not isomorphic.
\end{proof}

\section{Topo\-logical Euler characteristic and piecewise isomorphisms}
\label{sec.topEuler_characterstic_Grothendiecl_ring}
Throughout this section we assume that $\kk$ is algebraically closed.
	
	In the sequel we recall the definition and some basic facts of the topo\-logical Euler characteristic.
	In order to do this, we also recall some facts from \'etale $\ell$-adic cohomology with compact support.
	As a reference we take~\cite{Mi2013Lectures-on-Etale-} and~\cite{Mi1980Etale-cohomology}.

	Let $\ell$ be a prime number that is different from the characteristic of the ground field $\kk$.
	For a variety $X$, the group $H^i_c(X_{\et}, \QQ_\ell)$
	denotes the $i$-th \'etale $\ell$-adic cohomology with compact support, 
	i.e.
	\[
		H^i_c(X_{\et}, \QQ_\ell) \coloneqq  \left( \varprojlim H^i_c(X_{\et}, \ZZ / \ell^n \ZZ) \right) \otimes_{\ZZ_\ell} \QQ_{\ell} \, ,
	\]
	where $\ZZ_\ell \coloneqq \varprojlim \ZZ/ \ell^n \ZZ$ and $\QQ_\ell$ is the quotient field of $\ZZ_{\ell}$, 
	see e.g.~\cite[\S18, \S19]{Mi2013Lectures-on-Etale-}. We have:

	\begin{lemma}
		\label{Lem.Finitness_cohomology}
		Let $X$ be a variety. Then the $\QQ_{\ell}$-vector space 
		$H^i_c(X_{\et}, \QQ_{\ell})$ has finite dimension and vanishes for $i > 2 \dim X$.
	\end{lemma}
				
	\begin{proof}
		Let $k \colon  X \to \bar{X}$ be a completion. By definition
		(see e.g.~\cite[Definition~18.1]{Mi2013Lectures-on-Etale-}) we have
		\[
			H^i_c(X_{\et}, \ZZ / \ell^n \ZZ) = H^i(\bar{X}_{\et}, k_{!}(\ZZ / \ell^n \ZZ)) \, .
		\]
		By \cite[Theorem~19.2]{Mi2013Lectures-on-Etale-} the limit $\varprojlim H^i(\bar{X}_{\et}, k_{!}(\ZZ / \ell^n \ZZ))$ 
		is a finitely generated $\ZZ_{\ell}$-module; this implies the first statement.
		By \cite[Theorem~15.1]{Mi2013Lectures-on-Etale-}, we have that $H^i(\bar{X}_{\et}, k_{!}(\ZZ / \ell^n \ZZ))$
		vanishes for all $i > 2 \dim X$ and hence the second statement follows.
	\end{proof}

	The topological Euler-characteristic of a variety $X$ is defined by
	\[
		\chi(X) \coloneqq \sum_{i=0}^{2 \dim X} (-1)^i  \dim_{\QQ_\ell} H^i_c(X_{\et}, \QQ_\ell) \, .
	\] 
	
	The following properties are very useful in order to compute the topological Euler characteristic. 
	For lack of
	a reference with proof, we give an argument here:
	\begin{lemma}
		\label{Lem.Property_Eulercharacteristic}
		We have:
		\begin{enumerate}[leftmargin=*]
		\item \label{Lem.Property_Eulercharacteristic1} Let $X$ be a variety and let $Z \subseteq X$ be a closed subvariety. Then 
		$\chi(X) = \chi(X \setminus Z) + \chi(Z)$. 
		\item \label{Lem.Property_Eulercharacteristic2} Let $X, Y$ be varieties. Then $\chi(X \times Y) = \chi(X) \cdot \chi(Y)$.
		\end{enumerate}
	\end{lemma}

	\begin{proof}	
		\ref{Lem.Property_Eulercharacteristic1}: Let $U \coloneqq X \setminus Z$.
		By~\cite[Ch.XVII, \S5.1.16]{1973Theorie-des-topos-} we have a long exact sequence
		\begin{equation*}
		\begin{aligned}
		&\ldots \to H^{i}_c(U_{\et}, \ZZ/ \ell^n \ZZ) \to H^{i}_c(X_{\et}, \ZZ/ \ell^n \ZZ) \to H^{i}_c(Z_{\et}, \ZZ/ \ell^n \ZZ)  \\
		& \to H^{i+1}_c(U_{\et}, \ZZ/ \ell^n \ZZ) \to \ldots \, .
		\end{aligned}
		\end{equation*}
		Taking the limit over $n$ 
		and tensoring with $\QQ_{\ell}$ over $\ZZ_{\ell}$
		gives a long exact sequence
		\[
		\ldots \to H^{i}_c(U_{\et}, \QQ_{\ell}) \to H^{i}_c(X_{\et}, \QQ_{\ell}) \to H^{i}_c(Z_{\et}, \QQ_{\ell}) \to
		H^{i+1}_c(U_{\et}, \QQ_{\ell}) \to \ldots \, .
		\]
		The statement follows now by using Lemma~\ref{Lem.Finitness_cohomology}.
		
		\ref{Lem.Property_Eulercharacteristic2}: By the K\"unneth-formula 
			(see \cite[Ch.~VI, Corollary~8.23]{Mi1980Etale-cohomology})
			we get for $0 \leq k \leq \dim X + \dim Y$
			\[
			\dim_{\QQ_\ell} H^k_c((X \times Y)_{\et}, \QQ_\ell) 
			= \sum_{i=0}^k \dim_{\QQ_\ell} H^i_c(X_{\et}, \QQ_\ell) \otimes_{\QQ_{\ell}} H^{k-i}_c(Y_{\et}, \QQ_\ell) \, ,
			\]  
			which gives $\chi(X \times Y) = \chi(X) \cdot \chi(Y)$. 
	\end{proof}

	\begin{example}
		\label{Exa.Smooth_curve}
		Let $C$ be a smooth irreducible projective curve of genus $g$. Then 
		\[
			H^0(C, \QQ_\ell) = H^2(C, \QQ_\ell) = \QQ_{\ell} \quad \textrm{and} \quad
			H^1(C, \QQ_\ell) = \QQ_{\ell}^{2g} \, ,
		\]
		see~\cite[Proposition~14.2]{Mi2013Lectures-on-Etale-}. Using that $H^i(C, \QQ_\ell) = 0$ for all $i > 2$
		(see Lemma~\ref{Lem.Finitness_cohomology}), it follows that $\chi(C) = 2-2g$.
	\end{example}

	\begin{example}\label{Exa:ChiAmPm}
		Let $m \geq 0$. Then 
		\[
			H^i(\PP^m, \QQ_{\ell}) = \left\{\begin{array}{rl}
				\QQ_{\ell} \, , & \textrm{$0 \leq i \leq 2m$ and $i$ is even} \\
				0 \, , & \textrm{otherwise}
			\end{array}\right. \, ,
		\]
		see~\cite[Example~16.3]{Mi2013Lectures-on-Etale-}. Hence,
		$\chi(\PP^m) = m+1$. Using Lemma~\ref{Lem.Property_Eulercharacteristic} and that 
		$\AA^m \simeq \PP^m \setminus \PP^{m-1}$ for all $m \geq 0$ 
		(where  $\PP^{m-1}$ is linearly embedded in $\PP^m$ and
		$\PP^{-1} = \varnothing$), it follows that $\chi(\AA^m) = 1$ for all $m \geq 0$.
	\end{example}
	
	\begin{example}
		\label{Exa.A1bundle}
		Let $n \geq 0$ and let $\pi \colon X \to B$ be a locally trivial $\AA^n$-bundle with respect to the Zariski topology, where
		$B$ is any variety.
		Then $\chi(X) = \chi(B)$. Indeed, we proceed by induction on $\dim B$ and note that the
		case $\dim B = 0$ is clear. By assumption there exists a closed subvariety $B' \subseteq B$ 
		such that $\pi$ is trivial over $B \setminus B'$ and $\dim B' < \dim B$. Hence,
		\[
			\chi(X) = \chi(\AA^n \times (B \setminus B')) + \chi(\pi^{-1}(B')) = \chi(\AA^n \times B)
			= \chi(\AA^n) \chi(B) = \chi(B) \, ,
		\]
		where for the second equality  we used the induction hypothesis $\chi(\pi^{-1}(B')) = \chi(B') = \chi(\AA^n \times B')$.
	\end{example}

	Let us denote by $\Var_\k$ the set of isomorphism classes of varieties (over $\kk$)
	and denote by $\ZZ \Var_\k$ the free abelian group over $\Var_\k$. Moreover, let
	$I \subseteq \ZZ \Var_\k$ be the subgroup that is generated by
	\[
	[X] - [Z] - [X \setminus Z]
	\]
	for all varieties $X$ and closed subvarieties $Z$ (here $[W]$ denotes the isomorphism class of a variety $W$).
	By bilinear extension of the operation $[X] \cdot [Y] \coloneqq [X \times Y]$ we get a ring structure on 
	$\ZZ \Var_\k$ and $I$ is an ideal of it. The \emph{Grothendieck ring} is then the quotient of $\ZZ \Var_\k$ by $I$:
	\[
	K_0(\Var_\k) \coloneqq \ZZ \Var_\k /I  \, ,
	\] 
	see e.g. \cite[\S2.2.1]{LiSe2010The-Grothendieck-r}. 
	By abuse of notation we denote by $[X]$ the class of a variety $X$ in $K_0(\Var_\k)$.
	Using Lemma~\ref{Lem.Property_Eulercharacteristic}, the topological Euler characteristic
	gives a ring homomorphism
	\[
		\chi \colon 
		K_0(\Var_\k) \to \ZZ \, , 
		\quad \sum_{i=0}^n a_i [X_i] \mapsto \sum_{i=0}^n a_i \chi(X_i) .
	\]

\begin{definition}
	\label{Def.piecewise_isomorphic}
	Let $X, Y$ be varieties. Then $X, Y$ are called \emph{piecewise isomorphic} if~$X$ and $Y$
	can be decomposed as disjoint unions of $n \geq 1$ locally closed subsets $X_1, \ldots, X_n$ and 
	$Y_1, \ldots, Y_n$, respectively, such that $X_i \simeq Y_i$ for all $1 \leq i \leq n$. 
	Note that we take on $X_i$ the reduced scheme structure.
\end{definition}

Note that ``piecewise isomorphic'' defines an equivalence relation among the varieties.

\begin{remark}\label{Rem:PWirr}
In the above definition, if some $X_i$ is not irreducible, we replace $X_i$ in the decomposition by one of its irreducible components and its complement, and do the same for their images in $Y_i$. After finitely many such steps we may assume that all $X_i$ are irreducible, and thus all $Y_i$ are irreducible.
\end{remark}

	\begin{remark}
		If $X, Y$ are piecewise isomorphic, then $\dim X = \dim Y$. 
	\end{remark}

\begin{example}
Let $\Gamma\subseteq \p^2$ be an irreducible cuspidal cubic curve. Then $\Gamma$ is piecewise isomorphic to $\p^1$. Indeed, we choose $X_1,X_2\subseteq \Gamma$ as $X_1$ to be the singular locus (one point) and $X_2=\Gamma\setminus X_1$. Then, $X_1$ is isomorphic to $Y_1=[0:1]\in \p^1$ and $X_2$ is isomorphic to $Y_2=\p^2\setminus Y_1\simeq \A^1$.
\end{example}

\begin{lemma}
	\label{lemm.Grothendieckring}
	Let $X, Y$ be varieties and consider the following statements:
	\begin{enumerate}[leftmargin=*]
		\item \label{lemm.Grothendieckring1} $X$ and $Y$ are piecewise isomorphic;
		\item \label{lemm.Grothendieckring1.5}
		There are open subsets $U \subseteq X$, $V \subseteq Y$ such that
		\[
			U \simeq V\, ,\quad \dim X\setminus U=\dim Y\setminus V< \dim X=\dim Y
		\]
		and $X \setminus U$, $Y \setminus V$ are piecewise isomorphic.
		Moreover, if $X' \subseteq X$, $Y' \subseteq Y$ are locally closed subsets such that
		\[ 
			\dim X \setminus X', \dim Y \setminus Y' < \dim X = \dim Y 
		\]
		we may assume that $U \subseteq X'$ and $V \subseteq Y'$.
		\item \label{lemm.Grothendieckring3} $[X] = [Y]$ inside $K_0(\Var_\k)$;
		\item \label{lemm.Grothendieckring4} $\chi(X) = \chi(Y)$.
	\end{enumerate}
	Then we have $\ref{lemm.Grothendieckring1} \Leftrightarrow \ref{lemm.Grothendieckring1.5}
	\Rightarrow \ref{lemm.Grothendieckring3} \Rightarrow \ref{lemm.Grothendieckring4}$.
\end{lemma}

\begin{proof}
``$\ref{lemm.Grothendieckring1}\Rightarrow \ref{lemm.Grothendieckring1.5}$'': 
Let $X_1, \ldots X_n$ and $Y_1, \ldots Y_n$ be disjoint locally closed subsets of $X$ and $Y$, respectively, 
such that their union is equal to $X$ and $Y$, respectively, and such that there are isomorphisms
$\psi_i \colon X_i \iso Y_i$ for $1 \leq i \leq n$. 
By replacing $X_i$ with $X_i \cap X'$ and $X_i \setminus X'$ and analogously for $Y_i$ we may assume
that each $X_i$ is either contained in $X'$ or $X \setminus X'$ and 
$Y_i$ is either contained in $Y'$ or $Y \setminus Y'$.
Using Remark~\ref{Rem:PWirr}, we may moreover assume that all $X_i$ and $Y_i$ are irreducible.

 Let $d \coloneqq \dim X = \dim Y$. After reordering, we may assume that the closures $\overline{X_1}, \ldots, \overline{X_m}$ in $X$ are 
the irreducible components of $X$ of dimension $d$.
Then $\overline{Y_1}, \ldots, \overline{Y_m}$ are irreducible closed subsets of $Y$ of dimension $d$
and since $Y_1, \ldots, Y_m$ are disjoint and locally closed in $Y$, 
$\overline{Y_1}, \ldots, \overline{Y_m}$ are mutually different. This
implies that $\overline{Y}_1, \ldots, \overline{Y}_m$ are  the irreducible components of $Y$ of dimension $d$.
For $1 \leq i \leq m$, note that $X_i$ and $Y_i$ is contained in $X'$ and $Y'$, respectively, since $X \setminus X'$
and $Y \setminus Y'$ have strictly lower dimension that $X$ and $Y$.

For each $i\in\{1,\ldots,m\}$, the locally closed subset $X_i\subseteq X$ is equal to $X_i=A\cap \overline{X_i}$, where $A\subseteq X_i$ is open. Choosing $U_i'\subseteq X$ to be the open subset $A\setminus I$, where $I$ is the union of all 
 irreducible components of $X$ different from $\overline{X_i}$, we find a subset $U'_i \subseteq X_i$ that is open in $X$ with $\overline{U'_i} = \overline{X}_i$. By construction, $U_1', \ldots, U'_m$ are disjoint.
Similarly, there exist open disjoint subsets $V'_1, \ldots, V'_m$ in $Y$ such that
for all $1 \leq i \leq m$ we get $V'_i \subseteq Y'_i$ and $\overline{V_i'} = \overline{Y_i}$. For $1 \leq i \leq m$, let
\[
	U_i \coloneqq U'_i \cap \psi_i^{-1}(V'_i) \subseteq X \quad \textrm{and} \quad
	V_i \coloneqq \psi_i(U'_i) \cap V'_i \subseteq Y \, .
\] 
Then the restriction $\psi_i |_{U_i} \colon U_i \iso V_i$ is an isomorphism. 
Let $U \coloneqq \bigcup_{i=1}^m U_i$ and $V \coloneqq \bigcup_{i=1}^m V_i$.
By construction, $U \subseteq X'$ and $V \subseteq Y'$.
Since 
$U_1, \ldots, U_m$ are disjoint open subsets of $X$ and
$V_1, \ldots, V_m$ are disjoint open subsets of $Y$, we have $U \simeq V$.
Moreover, $X \setminus U$ and $Y \setminus V$ are piecewise isomorphic
(take the locally closed subsets $X_1 \setminus U_1, \ldots, X_m \setminus U_m, X_{m+1}, \ldots, X_n$
and $Y_1 \setminus V_1, \ldots, Y_m \setminus V_m, Y_{m+1}, \ldots, Y_n$ of $X$ and $Y$,
respectively). By construction, 
$\dim X \setminus U < \dim X$.

``$\ref{lemm.Grothendieckring1.5}\Rightarrow \ref{lemm.Grothendieckring1}$'': The decomposition of $X\setminus U$, together with $U$, gives the decomposition of $X$, and we do similarly for $Y$.

``$\ref{lemm.Grothendieckring1.5}\Rightarrow \ref{lemm.Grothendieckring3}$'':
By assumption we have $[U] = [V]$ inside $K_0(\Var_\k)$. 
As $X \setminus U$, $Y \setminus V$ are piecewise isomorphic and since
$\dim X \setminus U = \dim Y \setminus V < \dim X = \dim Y$ we may proceed by induction on $d$
in order to get $[X \setminus U] = [Y \setminus V]$ inside $K_0(\Var_\k)$. This implies~\ref{lemm.Grothendieckring3}.

``$\ref{lemm.Grothendieckring3}\Rightarrow\ref{lemm.Grothendieckring4}$'': Follows from Lemma~\ref{Lem.Property_Eulercharacteristic}.
\end{proof}
\begin{corollary}\label{coro:Equidim}
If $X$ and $Y$ are two varieties that are both equidimensional and piecewise isomorphic, then $X$ and $Y$ are birational to each other.
\end{corollary}
\begin{proof}
	Follows from the implication \ref{lemm.Grothendieckring1} $\Rightarrow$ \ref{lemm.Grothendieckring1.5} 
	in Lemma~\ref{lemm.Grothendieckring}.
\end{proof}\begin{example}
Corollary~\ref{coro:Equidim} really needs both $X$ and $Y$ to be equidimensional. Indeed, $\AA^1 \amalg \{pt\}$ and $\PP^1$ are piecewise isomorphic, but not birational.
\end{example}

\begin{lemma}\label{CurvepiecewiseBir}
For $i\in \{1,2\}$, let $\Gamma_i$ be a variety of  dimension $1$, and let $\Gamma_i'\subseteq \Gamma_i$ be the union of the irreducible components of $\Gamma_i$ of dimension $1$ $($we remove isolated points$)$. Then, the following are equivalent:
\begin{enumerate}
\item\label{CurvepiecewiseBir1}
$\Gamma_1$ and $\Gamma_2$ are piecewise isomorphic;
\item\label{CurvepiecewiseBir2}
$\Gamma_1'$ and $\Gamma_2'$ are birational, and  $\chi(\Gamma_1)=\chi(\Gamma_2)$.
\end{enumerate}
\end{lemma}

\begin{proof}
``$\ref{CurvepiecewiseBir1}\Rightarrow\ref{CurvepiecewiseBir2}$'': Suppose that $\Gamma_1$ and $\Gamma_2$ are piecewise isomorphic. The implication \ref{lemm.Grothendieckring1} $\Rightarrow$ \ref{lemm.Grothendieckring1.5} 
	in Lemma~\ref{lemm.Grothendieckring} implies that $\Gamma_1'$ and $\Gamma_2'$ are birational and the implication \ref{lemm.Grothendieckring1} $\Rightarrow$ \ref{lemm.Grothendieckring4}  yields $\chi(\Gamma_1)=\chi(\Gamma_2)$.

``$\ref{CurvepiecewiseBir2}\Rightarrow\ref{CurvepiecewiseBir1}$'': As $\Gamma_1'$ and $\Gamma_2'$ are birational, there are open dense subsets $U_1\subseteq \Gamma_1'$ and $U_2\subseteq \Gamma_2'$ such that $U_1\simeq U_2$. For each $i\in \{1,2\}$, the set $U_i$ is open in $\Gamma_i$, and $F_i=\Gamma_i\setminus U_i$ is finite. As $\chi(\Gamma_1)=\chi(\Gamma_2)$ and $\chi(U_1)=\chi(U_2)$,  Lemma~\ref{Lem.Property_Eulercharacteristic} implies that $\chi(F_1)=\chi(F_2)$. The two finite sets $F_1$ and $F_2$ are thus isomorphic, which proves that $\Gamma_1$ and $\Gamma_2$ are piecewise isomorphic.
\end{proof}

\begin{lemma}\ \label{Lem:UnionTree}
\begin{enumerate}
\item\label{UnionTree1}
Let $X$ be a disjoint union of trees of smooth projective rational curves, with $r$ irreducible components and $s$ connected components. Then, $X$ is piecewise isomorphic to the disjoint union of $r$ copies of $\A^1$ and $s$ points.
\item\label{UnionTree2}
Let $X$, $Y$ be two disjoint unions of trees of smooth projective rational curves. Then, $X$ and $Y$ are piecewise isomorphic if and only if they have the same number of irreducible components and the same number of connected components.
\end{enumerate}
\end{lemma}
\begin{proof}
\ref{UnionTree1}:
If $T$ is a tree of smooth projective rational curves, with $r$ irreducible components, it is piecewise isomorphic to the disjoint union of one point and $r$ copies of $\A^1$. Hence, if $X$ has $r$ irreducible components and $s$ connected components, then $X$ is piecewise isomorphic to the disjoint union of $r$ copies of $\A^1$ and $s$ points, and thus satisfies  $\chi(X)=r+s$ (Lemma~\ref{Lem.Property_Eulercharacteristic} and Example~\ref{Exa:ChiAmPm})

\ref{UnionTree2}: By Lemma~\ref{CurvepiecewiseBir}, $X$ and $Y$ are piecewise isomorphic if and only if they are birational and $\chi(X)=\chi(Y)$. The first assertion is equivalent to ask that $X$ and $Y$ have the same number of irreducible components. The above  calculation shows that the two assertions are equivalent to ask that $X$ and $Y$ have the same number of irreducible components and the same number of connected components.
\end{proof}

	\begin{lemma}
		\label{Lem.Example_piecewise_isom}
		Let $U \subseteq \PP^2$ be a dense open subset.
		For $i = 1, 2$, let
		\[
			X_i \coloneqq U \amalg S_i \amalg R_i  \, ,
		\]
		where $S_i$ is a disjoint union of $s_i \geq 0$ points and $R_i$ is the disjoint union of $r_i \geq 0$ irreducible curves.
		If $X_1$ and $X_2$ are piecewise isomorphic, then:
		\begin{enumerate}[leftmargin=*]
			\item $R_1$ and $R_2$ are birational; in particular $r_1 = r_2$.
			\item If moreover all irreducible components of $R_1$, $R_2$ are isomorphic, then $s_1 = s_2$.
		\end{enumerate}
	\end{lemma}
	
	\begin{proof}
		By Lemma~\ref{lemm.Grothendieckring} there exist open dense subsets $U_1, U_2 \subseteq U$
		that are isomorphic and such that $E_1 \coloneqq X_1 \setminus U_1$ and $E_2 \coloneqq X_2 \setminus U_2$ 
		are piecewise isomorphic.
		
		Let $(\PP^2 \setminus U_i)'$ and $(U \setminus U_i)'$ be the union of one-dimensional irreducible
		components of $\PP^2 \setminus U_i$ and $U \setminus U_i$, respectively.		
		The isomorphism $U_1 \iso U_2$ gives a birational map 
		$\p^2\dasharrow\p^2$ that decomposes into $n$ blow-ups and $n$ blow-downs. 
		Hence, $(\PP^2 \setminus U_1)'$ and $(\PP^2 \setminus U_2)'$ are birational and thus
		$(U \setminus U_1)'$ and $(U \setminus U_2)'$ are birational as well.
		By Lemma~\ref{CurvepiecewiseBir}, the unions of 
		one-dimensional irreducible components of $E_1$ and $E_2$  are birational 
		and
		since $E_i \coloneqq (U \setminus U_i) \amalg S_i \amalg R_i$ it follows that $R_1$ and $R_2$ 
		are birational. If the irreducible components of $R_1$ and $R_2$ are all isomorphic to $C$, we get 
		$\chi(E_i) = \chi(U) - \chi(U_i) + s_i  + r_i \chi(C)$ and  since
		$\chi(E_1) = \chi(E_2)$, $\chi(U_1) = \chi(U_2)$ we get $s_1 = s_2$.
	\end{proof}

\begin{example}\label{A2P2minus}
We show that Lemma~\ref{CurvepiecewiseBir} does not generalise to higher dimension.

The irreducible surfaces $X=\AA^2$ and $Y=\PP^2\setminus \{[1:0:0],[0:1:0]\}$ are birational and satisfy $\chi(X)=\chi(Y)=1$ (Lemma~\ref{Lem.Property_Eulercharacteristic} and Example~\ref{Exa:ChiAmPm}), but $X$ and $Y$ are not piecewise isomorphic. Indeed, this last assertion follows from Lemma~\ref{Lem.Example_piecewise_isom} 
and the fact that $X$ is piecewise isomorphic to $\AA^2 \setminus \{(0, 0)\} \amalg \{pt\}$ and 
$Y$ is piecewise isomorphic to $\AA^2 \setminus \{(0, 0)\} \amalg \AA^1$.
\end{example}
\begin{proposition}\label{prop:ComplementsCurves}
Let $X$, $Y$ be two varieties, that are piecewise isomorphic and let $E\subseteq X$, $F\subseteq Y$ be locally closed subvarieties of dimension at most~$1$. If $E$ and $F$ are piecewise isomorphic, then $X\setminus E$ and $Y\setminus F$ are piecewise isomorphic.
\end{proposition}
\begin{proof}
By definition, $X$ and $Y$
	can be decomposed as disjoint unions of $n \geq 1$ locally closed subsets $X_1, \ldots, X_n$ and 
	$Y_1, \ldots, Y_n$, respectively, such that there is an isomorphism $\varphi_i\colon X_i \iso Y_i$ for each $i\in \{1,\ldots,n\}$. 
	
	$\mathbf{(A)}$ We first consider the case where $E$ is one point (and thus $F$ is also one point). In this case, there is exactly one $i\in \{1,\ldots,n\}$ and one $j\in \{1,\ldots,n\}$ such that $E\in X_i$ and $F\subseteq Y_j$. If $i=j$, the isomorphism $\varphi_i$ induces an isomorphism from $X_i\setminus (E\cup \varphi_i^{-1} (F))$ to $Y_i\setminus (F\cup \varphi_i(E))$, and we are done if $\varphi_i(E)=F$ or otherwise use an isomorphism from the point $\varphi_i^{-1} (F)$ to the point $\varphi_i(E)$. 
	 If $i\neq j$, then  $\varphi_i$ gives an isomorphisms  $X_i\setminus E\iso Y_i\setminus \varphi_i(E)$ and $\varphi_j$ gives an isomorphism $X_j\setminus \varphi_{j}^{-1}(F)\iso Y_j\setminus F$. It then remains to use an isomorphism from the point $\varphi_j^{-1} (F)$ to the point $\varphi_i(E)$.
	
	$\mathbf{(B)}$ If $E$ (and thus $F$) is finite, we proceed by induction on the number of points, applying $\mathbf{(A)}$, and obtain the result.
	
	$\mathbf{(C)}$ In the general case, we proceed by induction on the number of irreducible components of $E$ of dimension $1$. When no such component exists, we are done by $\mathbf{(B)}$.  As $E$ and $F$ are piecewise isomorphic, there are open subsets $U'\subseteq E$ and $V'\subseteq F$ such that $E\setminus U'$ and $F\setminus V'$ are finite, with the same number of points, and there is an isomorphism $\psi\colon U'\iso V'$ (Lemma~\ref{lemm.Grothendieckring}). 
	We take an open subset $U\subseteq U'$, that is irreducible and infinite, take $V=\psi(U)\subseteq V'$, which is also open, irreducible and infinite, and obtain that $E\setminus U$ and $F\setminus V$ are piecewise isomorphic. 	
	As $U,V$ are irreducible curves, there is  exactly one $i\in \{1,\ldots,n\}$ and one $j\in \{1,\ldots,n\}$ such that $U\cap X_i$ and $V\cap  Y_j$ are infinite. 
	By removing finitely many points of $U$ and $V$ respectively, we may assume that $U\subseteq X_i$, $V\subseteq Y_j$ and still assume that $U$, $V$ are isomorphic (we remove points outside of $X_i$ or $Y_j$ and their images under $\psi$ or $\psi^{-1}$), and that $E\setminus U$, $F\setminus V$ are piecewise isomorphic. 
	We now prove that $X\setminus U$ and $Y\setminus V$ are piecewise isomorphic, which will achieve the proof, as $E\setminus U$ and $F\setminus V$ are piecewise isomorphic, with one irreducible component of dimension $1$ less.
	
	If $i=j$, the isomorphism $\varphi_i$ restricts to an isomorphism
	\[
		X_i\setminus (U\cup \varphi_i^{-1} (V)) \iso Y_i\setminus (V\cup \varphi_i(U)) \, ,
	\] 
	so we only need to see that $\varphi_i^{-1} (V)\setminus (U\cap \varphi_i^{-1} (V))$ and $\varphi_i(U)\setminus (V\cap \varphi_i(U))$ are piecewise isomorphic. Applying $\varphi_i$ to the first set, we need to show that $V\setminus (\varphi_i(U)\cap V)$ and $\varphi_i(U)\setminus (V\cap \varphi_i(U))$ are piecewise isomorphic. If $\varphi_i(U)\cap V$ is finite, this follows from $\mathbf{(B)}$. If $\varphi_i(U)\cap V$ is infinite, then $V\setminus (\varphi_i(U)\cap V)$ and $\varphi_i(U)\setminus (V\cap \varphi_i(U))$ are both finite, with the same number of points
	by using the topological Euler characteristic (Lemma~\ref{Lem.Property_Eulercharacteristic}).
	
	 If $i\neq j$, then  $\varphi_i$ gives an isomorphisms  $X_i\setminus U\iso Y_i\setminus \varphi_i(U)$ and $\varphi_j$ gives an isomorphism $X_j\setminus \varphi_{j}^{-1}(V)\iso Y_j\setminus V$. It then remains to use an isomorphism from the curve $\varphi_j^{-1} (V)$ to the point $\varphi_i(U)$.
\end{proof}

\begin{lemma}\label{Lem:SMoothRatDecomposition}
Let $X$ be a smooth projective rational surface with
Picard group $\Pic(X)\simeq \Z^n$. Then, there is an open subset $U\subseteq X$ isomorphic to $\A^2$, a point $p\in X$ and locally closed curves $C_1,\ldots,C_n\subseteq X$ such that each $C_i$ is isomorphic to $\A^1$ and $X$ is the disjoint union 
\[
	X=U \amalg C_1\amalg \cdots \amalg C_n \amalg  \{p\} \, .
\]
Moreover, for each finite set $\Delta\subseteq X$ and each closed curve $E\subseteq X$, we may choose the above decomposition such that $\Delta\subseteq U$, $\{p\}\cap (\Delta \cup E)=\varnothing$ and  $C_i\not\subseteq E$ for each $i\in \{1,\ldots,n\}$.
\end{lemma}
\begin{proof}
If $X=\p^2$, we choose a general line $\ell\subseteq \p^2$ and a general point $p\in \ell$, and let $U=\p^2\setminus \ell$, $C_1=\ell\setminus \{p\}$. This gives the statement in case $X = \PP^2$.

We may thus assume that $X$ is not isomorphic to $\p^2$, so there is a birational morphism $X\to \mathbb{F}_d$ for some Hirzebruch surface $\mathbb{F}_d$ where $d \geq 0$. 
We proceed by induction on the number of blow-ups of $X \to \FF_d$. If $X = \FF_d$, we choose a general section $\ell_1$
of $\FF_d \to \PP^1$ of self-intersection $d$ and a general fibre $\ell_2$ of $\FF_d \to \PP^1$ 
and let $p \coloneqq \ell_1 \cap \ell_2$, $U \coloneqq \FF_d \setminus (\ell_1 \cup \ell_2)$. 
Then, $C_i=\ell_i\setminus \{p\}\simeq \A^1$ for each $i\in \{1,2\}$.

We then do successive blow-ups, where we may always assume that the points blown-up are in the subset $U\simeq \A^2$. Then, taking a general line of $\A^2$ passing through the point, the strict transform of the line is a closed curve $C_i\simeq \A^1$ of the blow-up of $\A^2$, with complement isomorphic to $\A^2$.
\end{proof}

\begin{proposition}\label{prop:NormalTrees}
Let $Y$ be a normal
	projective rational surface, 
	that admits a desingularisation $X\to Y$ with  $\Pic(X)\simeq \Z^n$ 
	and exceptional divisor being a finite disjoint union of trees of smooth rational curves with $r\ge 0$
	irreducible components $($this holds for instance when $Y$ only has Du Val singularities$)$. 
	Then, $Y$ is piecewise isomorphic to the disjoint union of $\A^2$, one point, and $n-r$ copies of $\A^1$, 
	with $n-r\ge 1$.
\end{proposition}
\begin{proof}
If $Y$ is smooth, then the result, with $r=0$, directly follows from Lemma~\ref{Lem:SMoothRatDecomposition}. 
We then assume that $Y$ is singular. Let $E\subseteq X$ be the exceptional divisor of the desingularisation $X\to Y$,
 let $r$ be the number of irreducible components of $E$ and 
let $s$ be the number of connected components of $E$.
Then $s$ is equal to the number of points of the image of  $E$ under $X \to Y$,
since all fibres of $X \to Y$ are connected by the normality of $Y$, 
see Zariski's main theorem, \cite[Ch.~III, Corollary~11.4]{Ha1977Algebraic-geometry}.
Moreover, $r,s \geq 1$, since otherwise every fibre of the proper birational morphism $X \to Y$ onto the normal variety $Y$ 
would be a single point, i.e.~$X \to Y$ would be an isomorphism, see
\cite[Corollaire~4.4.9]{Gr1961Elements-de-geomet-III}.

We observe that $n-r\ge 1$. Indeed, the $r$ irreducible components $E_1, \ldots, E_r$ of $E$, together with the pull-back $P$ of an irreducible curve of $Y$ not passing through the image of $E$ under $Y \to X$ 
are $r+1$ effective divisors on $X$, that are linearly independent: 
Assume $D \coloneqq aP + \sum_{i=1}^r a_i E_i = 0$ in $\Pic(X)$ for some
$a, a_1, \ldots, a_r \in \ZZ$. Then $0 = L \cdot D = a L^2$, which shows that $a = 0$. Since the intersection matrix
$(E_i \cdot E_j)$ is negative definite (see \cite[\S2, (2.1)]{Ar1962Some-numerical-cri}), we get $a_1 = \ldots = a_r = 0$.

By Lemma~\ref{Lem:UnionTree}\ref{UnionTree1}, $E$ is piecewise isomorphic to the disjoint union of $r$ copies of $\A^1$ and $s$ points.
By Lemma~\ref{Lem:SMoothRatDecomposition}, $X$ is piecewise isomorphic to the disjoint union of $\A^2$ with one point and $n$ copies of $\A^1$.

Hence, Proposition~\ref{prop:ComplementsCurves} shows that $X\setminus E$ is piecewise isomorphic to the disjoint union of $\A^2\setminus \Delta$ with $n-r$ copies of $\A^1$, where $\Delta$ is a finite set of $s-1$ points.

Let $S$ be the image of $E$ under $X \to Y$.
As $Y\setminus S$ is isomorphic to $X\setminus E$ and $S$ contains $s$ points, this proves that 
$Y$ is piecewise isomorphic to the disjoint union of $\A^2$ with one point and $n-r$ copies of $\A^1$.
\end{proof}

\begin{example}\label{ExampleQuartics}
Let $f,g\in \k[x,y,z] \setminus \{0\}$ 
be homogeneous polynomials of degree $3$ and $4$, respectively, without common factor. 
Then, $V_{\PP^2}(f,g)$ consists of $r$ points, with $1\le r\le 12$.
We define $X\subseteq \p^3$ to be the quartic rational surface given by 
\[X=\{[w:x:y:z]\in \p^3\mid wf(x,y,z)=g(x,y,z)\},\]
and consider the open subset $X_f\subseteq X$ where $f$ is non-zero and its complement 
$F \coloneqq X \setminus X_f = V_{\PP^3}(f, g)$
The hypersurface $X$ is normal, since the singular locus is finite, 
see e.g.~\cite[Ch.~II, Proposition~8.23]{Ha1977Algebraic-geometry}.
We have an isomorphism 
\[
	X_f\iso \p^2_f \, , \quad [w:x:y:z]\mapsto [x:y:z] \, ,
\] 
and $F$ is the union of $r$ lines through $q=[1:0:0:0]$. Hence, $X$ is piecewise isomorphic to the disjoint union of $\p^2_f$ with one point and $r$ copies of $\A^1$. We now distinguish different cases for the zero locus 
$\Gamma \coloneqq V_{\PP^2}(f) \subseteq \PP^2$.

If $\Gamma$ consists of one, two or three lines meeting in one point, then $X$ is piecewise isomorphic to the disjoint union of $\A^2$ with one point and $s$ copies of $\A^1$, for some $s\le r$, similarly as in Proposition~\ref{prop:NormalTrees}. The same holds if $\Gamma$ is the union of a conic and a line, meeting at only one point.

If $\Gamma$ consists of three general lines, then $X$ is piecewise isomorphic to the disjoint union of $\A^2$ with \emph{two} points and $r-2$ copies of $\A^1$. It is then \emph{not} piecewise isomorphic to a smooth
projective rational surface. Indeed, this follows from Lemma~\ref{Lem.Example_piecewise_isom}, since 
a smooth projective rational surface is piecewise isomorphic to a disjoint union of $\AA^2$, 
\emph{one} point and some copies of $\AA^1$, see~Lemma~\ref{Lem:SMoothRatDecomposition}.
A similar result holds if $\Gamma$ is the union of a conic and a line meeting at two points.

If $\Gamma$ is a cuspidal cubic curve, it is piecewise isomorphic to a line, so $\p^2_f$ is piecewise isomorphic to $\A^2$ (Proposition~\ref{prop:ComplementsCurves}). Hence, this case is similar to the one of a line. If $\Gamma$ is a nodal cubic, it is piecewise isomorphic to $\A^1\setminus \{0\}$ and thus $X$ is  piecewise isomorphic to the disjoint union of $\A^2$ with \emph{two} points and $r$ copies of $\A^1$, and thus $X$ is again not piecewise isomorphic to a 
smooth rational projective surface.

If $\Gamma$ is a smooth irreducible cubic, this one is not rational.
Let $\PP^1 \subseteq \PP^2$ be a line that intersects $\Gamma$ in exactly one point $p \in \PP^2$
and let $\Gamma_0 \coloneqq \Gamma \setminus p$.
As $X$ is piecewise isomorphic to $\AA^2 \setminus \Gamma_0 \amalg \{pt\} \amalg \coprod_{i=1}^{r+1} \AA^1$,
it follows that $X$ is not piecewise isomorphic to a smooth rational projective surface $Y$, since such a $Y$
is piecewise isomorphic to 
$\AA^2 \setminus \Gamma_0 \amalg \{pt\} \amalg \Gamma_0 \amalg \coprod_{i=1}^s \AA^1$ for some $s \geq 1$, see Lemma~\ref{Lem.Example_piecewise_isom}.
\end{example}

\begin{example} In Example~\ref{ExampleQuartics}, we may choose $f$ to obtain a singular normal rational quartic $X\subseteq \p^3$ which is not piecewise isomorphic to a smooth rational projective surface, and thus which does not admit a desingularisation by trees of smooth rational curves (follows from Proposition~\ref{prop:NormalTrees}).

We now choose $f=xy(x+y)$ and $g\in \k[x,y,z]_4$ general, and check that this gives  an example of a normal rational quartic
\[
	X=\{[w:x:y:z]\in \p^3\mid wf(x,y,z)=g(x,y,z)\} \, ,
\]
which admits a desingularisation by trees of smooth rational curves, but that does not have Du Val singularities. For this, we consider the birational map 
\[
	\psi\colon \p^2\dasharrow X \, , [x:y:z]\mapsto [g(x,y,z):xf(x,y,z):yf(x,y,z):zf(x,y,z)] \, ,
\] 
and observe that it has exactly twelve base-points, being $V_{\p^2}(f,g)$. Denoting by $\eta\colon Y\to \p^2$ the blow-up of the twelve points, we obtain a birational morphism $\psi\circ \eta\colon Y\to X$, that contracts the strict transforms of the three lines defined by $V_{\p^2}(f)$. These curves are smooth rational curves on $Y$ with self-intersection $-3$, intersecting into a common point. This tree of smooth rational curves is contracted to the singular point of multiplicity three of~$X$.
\end{example}
\section{Partial  answers in low degree}

In this chapter, we give for certain cases affirmative answers to Questions~\ref{Quest2} and~\ref{Quest3}.
Moreover, over an algebraically closed field, 
we show for any irreducible hypersurface $H \subseteq \PP^n$ of degree $d$ 
that there exists an element of $\Aut(\PP^3 \setminus H)$ that does not extend to 
an automorphism of $\PP^3$
in case $d = 2$ or $n=d=3$ and $H$ is singular. 
This gives thus a negative answer to Question~\ref{Quest1} in the above mentioned cases.

\subsection{Degree two}
\label{Subsec.Degree2}
The goal of this subsection is to prove Theorem~\ref{Thm.Degree2}, which gives an
affirmative answer to Question~\ref{Quest2} (and thus~\ref{Quest3}) for
irreducible quadric hypersurfaces in $\PP^n_{\kk}$ over an algebraically closed field $\kk$.
For doing this, we count the $\FF_q$-rational points, where $\FF_q$ denotes the finite field with $q$
elements.

We start with a lemma that is certainly well-known to the specialists. For lack of a reference, we insert a proof.

\begin{lemma}
	\label{Lem:Induced_Isom}
	Let $\ZZ \subseteq A$, $\ZZ \subseteq B$ be finitely generated ring extensions such that $A$, $B$ 
	are integral domains.
	If there exists a field $K$ and a $K$-isomorphism between $K \otimes_{\ZZ} A$ and $K \otimes_{\ZZ} B$, then
	$\FF_{q} \otimes_{\ZZ} A$ and $\FF_{q} \otimes_{\ZZ} B$ are $\FF_{q}$-isomorphic
	for some prime power $q$.
\end{lemma}


\begin{proof}
	As $\ZZ$ is a principal ideal domain, it is a regular ring.
	Using that $A$, $B$ are integral domains, we get that $A$, $B$ are flat $\ZZ$-modules; see \cite[Proposition~9.7, Ch.~III]{Ha1977Algebraic-geometry}. Hence, every ring extension $R_1 \subseteq R_2$
	induces injections
	\[
	R_1 \otimes_{\ZZ} A \subseteq R_2 \otimes_{\ZZ} A \quad
	\textrm{and} \quad 
	R_1 \otimes_{\ZZ} B \subseteq R_2 \otimes_{\ZZ} B  \, .
	\]
	
	Let $a_1, \ldots, a_n \in A$ with $A = \ZZ[a_1, \ldots, a_n]$ and
	$b_1, \ldots, b_m \in B$ with $B = \ZZ[b_1, \ldots, b_m]$. Denote by 
	$\varphi \colon K \otimes_{\ZZ} A \xlongrightarrow{\sim} K \otimes_{\ZZ} B$ a $K$-isomorphism.
	There exist
	finitely many $c_1, \ldots, c_r \in K$ such that the $\ZZ$-algebra $R$ spanned by
	$c_1, \ldots, c_r$ inside $K$ satisfies $\varphi(a_i) \in  R \otimes_{\ZZ} B$ 
	for all $1 \leq i \leq n$ and $\varphi^{-1}(b_j) \in  R \otimes_{\ZZ} A$ for all $1 \leq j \leq m$.
	In particular, $\varphi$ restricts to an $R$-isomorphism $R \otimes_{\ZZ} A \simeq R \otimes_{\ZZ} B$.
	We may assume that $R$ is non-zero and thus contains a maximal ideal $\mm$. Then we get an $R/\mm$-isomorphism
	\[
	(R/\mm) \otimes_{\ZZ} A \xlongrightarrow{\sim} (R/\mm) \otimes_{\ZZ} B \, .
	\]
	Denote by $F$ the prime field of $R/\mm$, i.e.~the 
	quotient field of the image of the unique ring homomorphism $\ZZ \to R/\mm$. 
	As $R$ is finitely generated as a $\ZZ$-algebra,
	$R/\mm$ is finitely generated as  an $F$-algebra 
	and by Noether's normalization theorem, the field extension $F \subseteq R /\mm$ is finite. By 
	the Artin-Tate lemma (see \cite[Proposition~7.8]{AtMa1969Introduction-to-co})
	it follows that $F$ is a finitely generated $\ZZ$-algebra. Hence, $F$ cannot be isomorphic to $\QQ$
	and thus $F = \FF_p$ where $p = \car(R/\mm)$. As the field extension $F \subseteq R/\mm$ is finite, 
	$R/\mm = \FF_{p^r}$ for some $r \geq 1$.
\end{proof}

\begin{corollary}
	\label{Cor:Same_number_of_points}
	Let $X, Y$ be affine integral schemes of finite type over $\ZZ$
	such that $X \to \Spec(\ZZ)$, $Y \to \Spec(\ZZ)$ are dominant.
	If the pull-backs $X_K$, $Y_K$ are $K$-isomorphic for some field $K$, then
	there exists a prime power $q$ and a bijection
	between the $\FF_{q}$-rational points $X(\FF_{q}) \to Y(\FF_{q})$
	that maps the regular $\FF_{q}$-rational points of $X$ bijectively
	onto the regular $\FF_{q}$-rational points of $Y$.
\end{corollary}

\begin{proof}
	This is a direct consequence of Lemma~\ref{Lem:Induced_Isom}.
\end{proof}

In the following lemmas, we denote by $Z(\FF_q)$ ($Z(\FF_q)_{\reg}, Z(\FF_q)_{\sing}$) the set of (regular, singular) 
$\FF_q$-rational points of a scheme $Z$ that is defined over $\ZZ$. Note that $Z(\FF_q)_{\reg} = Z(\FF_q) \setminus Z(\FF_q)_{\sing}$.

In case $Z$ is a hypersurface in $\AA^{n+1}_{\ZZ}$ given by a polynomial $f \in \ZZ[x_0, \ldots, x_n]$,
the set $Z(\FF_q)$ consists of those points in $(a_0, \ldots, a_n) \in \FF_q^{n+1}$ that satisfy
$f(a_0, $\ldots$, a_n) = 0$ and $(a_0, \ldots, a_n) \in Z(\FF_q)$ is regular, if
not all the partial derivatives 
$\frac{\partial f}{\partial x_0}, \ldots, \frac{\partial f}{\partial x_n} \in \ZZ[x_0, \ldots, x_{n}]$ vanish at $(a_0, \ldots, a_n)$.

\begin{lemma}
	\label{Lem:Number_of_F_q-points_1}
	For all integers $m,n$ with $1\le 2m-1\le n$, let $X_{m, n}$ be the hypersurface in $\AA^{n+1}_{\ZZ}$ given by 
	\[
	\sum_{i=0}^{m-1} x_{2i} x_{2i+1}  = 1.
	\]
	Then for any prime power $q$ we get
	\[
	\# X_{m, n}(\FF_q) = \# X_{m, n}(\FF_q)_{\reg} = q^{n-m}(q^m-1) \, .
	\] \end{lemma}\begin{proof}
	As $X_{m,n}$ is smooth over $\Spec(\ZZ)$, $X_{m, n}(\FF_q)_{\reg} = X_{m, n}(\FF_q)$.
	We proceed now by induction on $m \geq 1$. For any $n \geq 1$, we get
	$\# X_{1, n}(\FF_q) = q^{n-1} (q-1)$, as  $X_{1, n}$ and $(\AA_{\ZZ}^1 \setminus \{0\}) \times \AA_{\ZZ}^{n-1}$
	are isomorphic over $\ZZ$. Moreover, for $m > 1$:
	\[
	X_{m, n}(\FF_q) = V_{X_{m, n}}(x_{2m-1})(\FF_q) \coprod (X_{m, n})_{x_{2m-1}}(\FF_q) \, .
	\]
	As $(X_{m, n})_{x_{2m-1}}$ and $(\AA_{\ZZ}^1 \setminus \{0\}) \times \AA_{\ZZ}^{n-1}$ are isomorphic over $\ZZ$,
	it follows that $\# (X_{m, n})_{x_{2m-1}}(\FF_q) = q^{n-1}(q-1)$. As $V_{X_{m, n}}(x_{2m-1})$ and 
	$X_{m-1, n-1}$ are isomorphic over $\ZZ$, we get by induction
	\begin{eqnarray*}
		\# X_{m, n}(\FF_q) &=& \# X_{m-1, n-1}(\FF_q) + q^{n-1}(q-1) \\
		&=& q^{n-m}(q^{m-1}-1) + q^{n-1}(q-1) = q^{n-m}(q^m-1) \, . \qedhere
	\end{eqnarray*}
\end{proof}

\begin{lemma}
	\label{Lem:Number_of_F_q-points_2}
	For all integers with $0\le 2m\le n$,
	let $Y_{m, n}$ be the hypersurface in $\AA^{n+1}_{\ZZ}$ that is given by 
	\[
	\left(\sum_{i=0}^{m-1} x_{2i} x_{2i+1} \right) + x_{2m}^2 = 1.
	\]
	Then for every \emph{odd} prime power $q$ we get
	\[
	\# Y_{m, n}(\FF_q) = \# Y_{m, n}(\FF_q)_{\reg} = q^{n-m}(q^m+1) \, ,
	\] 
	and for every \emph{even} prime power $q$, we have
	\[
	\# (Y_{m, n})(\FF_{q}) = q^n \quad
	\textrm{and} \quad \# (Y_{m, n})(\FF_{q})_{\reg} = q^{n-2m}(q^{2m}-1)  \, .
	\]
\end{lemma}
\begin{proof}
	Suppose first that $q$ is odd. Here, $Y_{m,n}$ is smooth over $\Spec(\ZZ) \setminus \{(2)\}$, 
	so  $Y_{m, n}(\FF_q)_{\reg} = Y_{m, n}(\FF_q)$.
	We proceed by induction on $m \geq 0$. For any $n \geq 1$, we get $\# Y_{0, n} = 2 q^{n}$,
	as $Y_{0, n}$ is given by $x_0^2-1$ in $\AA^{n+1}_{\ZZ}$ over $\ZZ$. Assume $m > 0$.
	As in Lemma~\ref{Lem:Number_of_F_q-points_1}, we count the points where $x_{2m-1}=0$ and the points where $x_{2m-1}\neq 0$, and obtain:
	\begin{eqnarray*}
		\# Y_{m, n}(\FF_q) &=& \# Y_{m-1, n-1}(\FF_q) + q^{n-1}(q-1) \\
		&=& q^{n-m}(q^{m-1}+1) + q^{n-1}(q-1) = q^{n-m}(q^m + 1) \, .
	\end{eqnarray*}	
	
	Suppose now that $q$ is a power of $2$. As $\FF_{q}\to \FF_q, t\to t^2$ is bijective, the projection $Y_{m,n}\to \AA_{\ZZ}^n$, $(x_0, \ldots, x_n) \mapsto (x_0, \ldots, x_{2m-1}, x_{2m+1}, \ldots, x_n)$ induces a bijection $Y_{m,n}(\FF_q)\iso \AA_{\ZZ}^n(\FF_q)$, 
	which yields $\# (Y_{m, n})(\FF_{q}) = q^n $.
	The set of singular $\FF_q$-rational points of $Y_{m, n}$ is equal to
	\[
	(Y_{m, n})(\FF_{q})_{\sing} = \set{ (0, \ldots, 0, 1, x_{2m+1}, \ldots, x_{n} ) }{
		x_{2m+1}, \ldots, x_{n} \in \FF_{q} }
	\] 
	and hence $(Y_{m, n})(\FF_{q})_{\sing} = q^{n-2m}$. This gives the result.
\end{proof}
\begin{proposition}\label{prop:QuadricProj}
	If $\k$ is algebraically closed, every 
	irreducible quadric hypersurface in $\PP^n$ is given, under a suitable change of coordinates, by
	\[
	\sum_{i=0}^{m-1} x_{2i} x_{2i+1} \quad \textrm{or} \quad
	\left( \sum_{i=0}^{m-1} x_{2i} x_{2i+1} \right) + x_{2m}^2 \, ,
	\]
	where $1 < 2m-1 \leq n$ in the first case and $0 < 2m \leq n$ in the second case.
\end{proposition}
\begin{proof}	This follows from \cite[Theorem 1.8]{Pf1995Quadratic-forms-wi} if $\car(\k) \neq 2$ and from
	\cite[Theorem 4.3]{Pf1995Quadratic-forms-wi} if $\car(\kk) = 2$.
\end{proof}

\begin{proof}[Proof of Theorem~\ref{Thm.Degree2}]
	Let 
	\[
	\mathcal{M} = \mathcal{M}_X \coprod \mathcal{M}_Y
	\]
	where
	\[
	\mathcal{M}_X \coloneqq \set{X_{m, n}}{1 < 2m-1 \leq n} \quad \textrm{and} \quad 
	\mathcal{M}_Y \coloneqq \set{Y_{m, n}}{0 < 2m \leq n}
	\]
	and $X_{m, n}$, $Y_{m, n}$ are defined in Lemmas~\ref{Lem:Number_of_F_q-points_1} and \ref{Lem:Number_of_F_q-points_2}, respectively. Using Proposition~\ref{Prop.cyclic_covering}, Remark~\ref{Rem.mu_equal_1} and Proposition~\ref{prop:QuadricProj},
	it is enough to show that distinct elements of $\mathcal{M}$ are non-isomorphic over $\k$. 
	By Corollary~\ref{Cor:Same_number_of_points}, an isomorphism over $\kk$ between two elements from $\mathcal{M}$ would give a bijection between the $\FF_q$-rational points and between the regular $\FF_q$-rational points, for some prime power $q$.
	
	We consider for every prime power $q$ the map
	\[
	\Phi_q \colon \mathcal{M} \to \NN_0 \, , \quad
	Z \mapsto \# Z(\FF_q)_{\reg} \, .
	\]
	By Lemmas~\ref{Lem:Number_of_F_q-points_1}-\ref{Lem:Number_of_F_q-points_2}, we get
	\[
	\Phi_q(X_{m, n}) = q^{n-m}(q^m-1) \quad \textrm{for all prime powers $q$}
	\]
	and
	\[	
	\Phi_q(Y_{m, n}) = 
	\left\{\begin{array}{rl}
		q^{n-m}(q^m+1) & \textrm{if $q$ is an odd prime power}  \\
		q^{n-2m}(q^{2m}-1) &  \textrm{if $q$ is a power of $2$}
	\end{array}
	\right. \, .
	\]
	Hence, for odd prime powers $q$, the map $\Phi_q$ is injective.
	If $q = 2^r$ for some $r \geq 1$, then the restrictions $\Phi_{2^r} |_{\mathcal{M}_X}$ and 
	$\Phi_{2^r} |_{\mathcal{M}_Y}$ are still injective. By the Lemmas~\ref{Lem:Number_of_F_q-points_1}-\ref{Lem:Number_of_F_q-points_2}, all $\FF_{2^r}$-rational points of $X_{m, n}$
	are regular, whereas $Y_{m, n}$ has singular $\FF_{2^r}$-rational points. This achieves to prove
	that distinct elements from $\mathcal{M}$ are non-isomorphic over $\k$.
\end{proof}

The following example shows that Question~\ref{Quest1} has a negative answer for degree $2$ and any $n \geq 2$:
\begin{example}
	\label{Exa.Counterexample_Q1deg2}
	Let $n \geq 2$. For each $f\in \k[x_3,\ldots,x_{n}]$, homogeneous of degree~$2$ (we choose $f=0$ if $n=2$), 
	the following morphism defines a $\GG_a$-action
	\[
	\begin{array}{rcl}
		\rho \colon \GG_a \times \PP^n_{x_0x_1+x_2^2+f} &\to& \PP^n_{x_0x_1+x_2^2+f} \\
		(t, [x_0:x_1: \cdots: x_n]) &\mapsto& [x_0: x_1 -2tx_2 - t^2x_0: x_2 + tx_0: x_3:\cdots:x_{n}].
	\end{array}
	\]
	Since $\frac{x_0^2}{x_0x_1+x_2^2+f}$ is a $\GG_a$-invariant function on $\PP^n_{x_0x_1+x_2^2+f}$,
	it follows that for every non-constant univariate polynomial $q$ the following
	map
	\[
	[x_0:x_1: \cdots: x_n] \mapsto \rho \left(q\left(\frac{x_0^2}{x_0x_1+x_2^2+f}\right) , 
	[x_0:x_1: \cdots: x_n] \right) 
	\]
	defines an automorphism of $\PP^n_{x_0x_1+x_2^2+f}$ that doesn't extend to an element of $\Aut(\PP^n)$.

	Similarly, for each $n\ge 3$ and each polynomial $g\in \k[x_4,\ldots,x_n]$ (again $g=0$ if $n=3$),
	the following morphism defines a $\GG_a$-action
	\[
	\begin{array}{rcl}
		\rho \colon \GG_a \times \PP^n_{x_0x_1+x_2x_3+g} &\to& \PP^n_{x_0x_1+x_2x_3+g} \\
		(t, [x_0:x_1: \cdots: x_n]) &\mapsto& [x_0:x_1-tx_3: x_2+tx_0:x_3:\cdots: x_n].
	\end{array}
	\]
	Since $\frac{x_0^2}{x_0x_1+x_2x_3+g}$ is a $\GG_a$-invariant function on $\PP^n_{x_0x_1+x_2x_3+g}$,
	it follows that for every non-constant univariate polynomial $q$ the following
	map
	\[
	[x_0:x_1: \cdots: x_n] \mapsto \rho \left(q\left(\frac{x_0^2}{x_0x_1+x_2 x_3 + g}\right) , 
	[x_0:x_1: \cdots: x_n] \right) 
	\]
	defines an automorphism of $\PP^n_{x_0x_1+x_2x_3+g}$ that doesn't extend to an element of $\Aut(\PP^n)$. 
	
	If $\k$ is algebraically closed, it follows from Proposition~\ref{prop:QuadricProj} that every irreducible quadric is of one of the above form, up to change of coordinates.
\end{example}

\subsection{Degree three and beyond in dimension three}
\label{Subsec.Degree3n3}

Throughout this subsection we assume always that $\kk$ is algebraically closed.

Recall that an irreducible hypersurface in $\PP^3$ is normal if and only if its singular locus is 
finite \cite[Ch.~II, Proposition~8.23]{Ha1977Algebraic-geometry}.

\begin{remark}
	\label{Rem.Desingularization_is_blow-up}
	Let $X \subseteq \PP^2$ be a normal rational irreducible cubic hypersurface. Then
	there exist six (possibly infinitely near) points $p_1, \ldots, p_6$ in $\PP^2$
	such that the blow-up $\pi \colon \tilde{X} \to \PP^2$ of them satisfies:
	There exists a birational morphism $\eta \colon \tilde{X} \to X$ and it is the minimal desingularisation of $X$.
	In case $X$ is singular this follows from the proof of Lemma~\ref{Lem.Normal_cubic_surf}
	and in case $X$ is smooth this follows from~\cite[Ch.~V, Corollary~4.7, Remark~4.7.1]{Ha1977Algebraic-geometry}. Moreover in the smooth case $p_1, \ldots, p_6$
	belong to $\PP^2$ (and are in general position).
\end{remark}
\begin{lemma}
	\label{Lem.Non-rational_iff_cone_over_elliptic_curve}
	An  irreducible cubic hypersurface in $\PP^3$ is non-rational if and only if it is the cone
	over a smooth cubic curve in $\PP^2$. 
\end{lemma}	

\begin{proof}
	Let $X \subseteq \PP^3$ be a non-rational irreducible cubic hypersurface.
	Then $X$ has a singularity (as otherwise it would be the blow-up of $\PP^2$ at six points, see Remark~\ref{Rem.Desingularization_is_blow-up}), 
	say at $P = [1:0:0:0] \in \PP^3$ and it
	is given by a homogeneous polynomial of the form $f_2(x, y, z) w + f_3(x, y, z)$, 
	where $[w: x:y:z]$ denote the homogeneous coordinates of $\PP^3$ and
	$f_i \in \kk[x, y,  z]$ is homogeneous of degree $i$. As $X$ is non-rational, we get
	$f_2 = 0$ (otherwise the projection to $x,y,z$ gives a birational map to $\PP^2$) and $X$ is the cone in $\PP^3$ over a cubic curve $\Gamma\subset\PP^2$ given by $f_3=0$. As $X$ is irreducible, so is $\Gamma$. Then, $X$ is birational to $\Gamma\times \PP^1$, and thus $\Gamma$ is non-rational as $X$ is. This implies that $\Gamma$ is smooth, and that $X$ is the cone in $\PP^3$ over a smooth cubic curve in $\PP^2$.

	Assume now that $X$ is a cone over a smooth cubic curve $C$. 
	Blowing up the vertex yields a projective smooth ruled surface $S$ over 
	$C$. In particular $H^1(S, \mathcal{O}_S)$ doesn't vanish 
	(see \cite[Ch.~V, Corollary~2.5]{Ha1977Algebraic-geometry}) and hence $X$ is not rational
	(see \cite[Theorem~3.2]{KoSmCo2004Rational-and-nearl}).
\end{proof}	

The next lemma is certainly known to the specialists. For lack of a proof in any characteristic we provide one:

\begin{lemma}
	\label{Lem.Normal_cubic_surf}
	Let $X \subseteq \PP^3$ be a normal rational  cubic hypersurface.
	Then $X$ has only Du Val singularities and is piecewise isomorphic to the disjoint union of $\A^2$, one point, and $n$ copies of $\A^1$, with $1\le n\le  7$. Moreover, $3 \leq \chi(X) = 2 + n \leq 9$
	and $\chi(X) = 9$ if and only if $X$ is smooth.
\end{lemma}

\begin{proof}
	If $X$ is smooth, then $X$ is the blow-up of $6$ points of $\p^2$ (Remark~\ref{Rem.Desingularization_is_blow-up}) and is thus piecewise isomorphic to the disjoint
	union of $\AA^2$, one point and seven copies of $\AA^1$ by Lemma~\ref{Lem:SMoothRatDecomposition}. 
	In particular, $\chi(X)=9$.
	
	Thus we may assume that $X$ is singular and after some coordinate change we may assume that $X$ is given by 
	$f_2(x, y, z)w + f_3(x, y, z)$, where
	$f_2, f_3 \in \kk[x, y, z]$ are homogeneous polynomials of degree $2, 3$, respectively,
	and $[w:x:y:z]$ denote the homogeneous coordinates in $\PP^3$. If $f_2 = 0$, then, 
	$X$ is a cone over $V_{\p^2}(f_3)$. As $X$ is a rational cubic, $f_3$ is irreducible and $V_{\p^2}(f_3)$ has to be singular (Lemma~\ref{Lem.Non-rational_iff_cone_over_elliptic_curve}), contradicting the fact that $X$ is normal.
	Hence, $f_2 \neq 0$.
	
	The projection $X \setminus \{[1:0:0:0] \} \to \PP^2$, $[w:x:y:z] \mapsto [x:y:z]$ 
	gives a birational morphism
	whose inverse 	$\Phi \colon \PP^2 \dashrightarrow X$ is given by
	\[
	[x:y:z]\mapsto [f_3(x,y,z): x f_2(x,y,z): y f_2(x,y,z): z f_2(x,y,z)] \, .
	\]
	By \cite[Proposition~1.5]{Sakamaki:2010vx}, there exists a smooth projective surface $\tilde{X}$ which 
	is the blow-up $\pi \colon \tilde{X} \to \PP^2$ of six (possibly infinitely near) points
	in the zero locus $S = V(f_2, f_3)$ such that $\eta \coloneqq \Phi \circ \pi \colon \tilde{X} \to X$ is a morphism
	and $\eta$ is given by the complete linear system
	$|-K_{\tilde{X}}|$, where $K_{\tilde{X}}$ denotes the canonical divisor in $\tilde{X}$.
	Let $X_{\reg}$ be the open subvariety of regular points in $X$ and let $X_{\sing}$ be its complement in $X$.
	
		By applying the adjunction formula to the smooth closed hypersurface $X_{\reg}$ in $\PP^3 \setminus X_{\sing}$, 
		we get that the canonical divisor $K_X$ of $X$ is equal to $-H$, where $H$ is a hyperplane section of  $X\subseteq \p^3$.
		This gives $\eta^{\ast} K_X = K_{\tilde{X}}$.
		By~\cite[Theorem~2.7]{Ar1962Some-numerical-cri} it follows that the exception locus of $\eta$
		is a finite union of trees and that in fact $X$ has only Du Val singularities.

	By Proposition~\ref{prop:NormalTrees}, $X$ is piecewise isomorphic to a disjoint
	union of $\AA^2$, one point and $7-r$ copies of $\AA^1$, with $1\le r\le 6$. 
	In particular, $3\le \chi(X) \le 8$.
\end{proof}

We will use the following classification of non-normal irreducible cubic surfaces in $\PP^3$:

\begin{proposition}[{see \cite[Theorem~3.1]{LePaSc2011On-the-classificat}}]
	\label{Prop.Classification_of_non-normal_cubic_surfaces}
	Let $X \subseteq \PP^3$ be a non-normal irreducible cubic surface. Denote by $[w:x:y:z]$
	the homogeneous coordinates in $\PP^3$. 
	
	\begin{itemize}[leftmargin=*]
		\item If $X$ is a cone over a curve in $\PP^2$,
		then up to a coordinate change of $\PP^3$, the hypersurface $X$ is given by
		\[
		\label{Prop.Classification_of_non-normal_cubic_surfaces1}
		\tag{$\ast$}
		x^2 w + y^3 \quad \textrm{or} \quad 
		x^2 w + y^3 + xy^2 \quad \textrm{or} \quad xyw + x^3 + y^3 \, .
		\]
		Moreover, if $\car(\kk) \neq 3$, then $x^2 w + y^3$ and $x^2 w + y^3 + xy^2$ 
		are the same up to a linear coordinate change in $w, x, y$.
		
		\item If $X$ is not a cone over a curve in $\PP^2$, then up to a coordinate change of $\PP^3$, the hypersurface $X$ is given by
		\[
		\label{Prop.Classification_of_non-normal_cubic_surfaces2} 
		\tag{$\ast\ast$}	
		x^2w+y^2z \quad \textrm{or} \quad 
		xyw +y^2z + x^3 \quad \textrm{or} \quad
		xyw + (x^2+y^2)z  \, .
		\]
		Moreover, if $\car(\kk) \neq 2$, then $x^2w+y^2z$ and $xyw + (x^2+y^2)z$ are the same up to a linear coordinate 
		change in $w, x, y, z$.
	\end{itemize}
	In particular, $X$ always contains a line with multiplicity $2$ $($the line given by $x=y=0$ in the above equations$)$.
\end{proposition}

In the following lemma, we give a similar decompositions into locally closed subsets as in Lemma~\ref{Lem.Normal_cubic_surf}
for irreducible non-normal cubic hypersurfaces in $\PP^3$:

\begin{lemma}
	\label{Lem.Non-Normal_cubic_surf} 
	Let 
	\[
	\begin{aligned}
		f_1 &= xyw + x^3 + y^3  \, ,  & 
		f_2 &= x^2 w + y^3  \, ,  &
		f_3 &= x^2 w + y^3 + xy^2  \, ,   \\
		f_4 &= xyw +y^2z + x^3  \, ,  &
		f_5 &= x^2w+y^2z  \, ,  &
		f_6 &= xyw + (x^2+y^2)z  \, .
	\end{aligned}
	\]
	Then, $X_i=V_{\PP^3}(f_i)\subseteq \PP^3$ is piecewise isomorphic to the disjoint union of $\A^2$, one point, and $n_i$ copies of $\A^1$, with $n_1=0$, $n_2=n_3=n_4=1$ and $n_5=n_6=2$. In particular, $\chi(X_i)=2+n_i\in \{2,3,4\}$.
\end{lemma}

\begin{proof}
	We can decompose $X_i$ into  $X_i \cap \PP^3_x$ and $X_i \cap V_{\PP^3}(x)$.
	Then we get the following table
	\begin{center}
		\begin{tabular}{l|l|l}
			Equation & $X_i \cap \PP^3_x$ & $X_i \cap V_{\PP^3}(x)$ \\
			\hline
			$xyw + x^3 + y^3$ & $\AA^1 \times (\AA^1 \setminus \{0\})$ & $\PP^1$ \\
			$x^2 w + y^3$ & $\AA^2$ & $\PP^1$ \\
			$x^2 w + y^3 + xy^2$ & $\AA^2$ & $\PP^1$ \\
			$xyw +y^2z + x^3$ & $\AA^1 \times (\AA^1 \setminus \{0\})$ & $\PP^1 \vee \PP^1$ \\
			$x^2w+y^2z$ & $\AA^2$ & $\PP^1 \vee \PP^1$  \\
			$xyw + (x^2+y^2)z$ & $\AA^2$ & $\PP^1 \vee \PP^1$ \\
		\end{tabular}
	\end{center}
	where $\PP^1 \vee \PP^1$ denotes two copies of $\PP^1$ in $\PP^2$ 
	(that intersect transversally in exactly one point). This implies the statement.
	%
\end{proof}
\begin{corollary}\label{Coro:ChiRatCubic}
	Let $X$ be a rational cubic hypersurface of $\p^3$. Then, $X$ is piecewise isomorphic to the disjoint union of $\A^2$, one point, and $n$ copies of $\A^1$, with $0\le n\le  7$. Moreover, $2\le \chi(X)=2+n\le 9$.
\end{corollary}
\begin{proof}
	If $X$ is normal, this follows from Lemma~\ref{Lem.Normal_cubic_surf}. Otherwise, this
	follows from Lemma~\ref{Lem.Non-Normal_cubic_surf} in combination with Proposition~\ref{Prop.Classification_of_non-normal_cubic_surfaces}.
\end{proof}

\begin{lemma}
	\label{Lem.cubic_Euler_char}
	Let $X \subseteq \PP^3$ be an irreducible cubic hypersurface. 
	If $X$ is rational, then $\chi(X)\ge 2$. If $X$ is not rational, then $\chi(X)=1$.
\end{lemma}

\begin{proof}
If  $X$ is rational, then Corollary~\ref{Coro:ChiRatCubic} gives $\chi(X)\ge 2$.
	
	Assume now that $X$ is non-rational. 
	By Lemma~\ref{Lem.Non-rational_iff_cone_over_elliptic_curve}, 
	$X$ is the cone over a smooth cubic curve $C$ in $\PP^2$.
	Let $p \in X$ be the unique singularity. Then $X \setminus \{p\}$ is a locally trivial $\AA^1$-bundle over $C$
	(with respect to the Zariski topology). Hence $\chi(X) = 1+\chi([X \setminus p]) = 1 + \chi(C) = 1$,
	where the second and third equality follow from Example~\ref{Exa.A1bundle} 
	and Example~\ref{Exa.Smooth_curve}, respectively.
\end{proof}

\begin{proposition}
	\label{Prop.Grothendieckring}
	Let $X_1, X_2$ be normal projective rational surfaces having only 
	singularities
	that can be solved by trees of smooth rational curves $($e.g.~Du Val singularities$)$.
	Then the following statements are equivalent
	\begin{enumerate}[leftmargin=*]
		\item \label{Prop.Grothendieckring1} $X_1$ and $X_2$ are piecewise isomorphic;
		\item \label{Prop.Grothendieckring2} $[X_1] = [X_2]$ inside $K_0(\Var_\k)$;
		\item \label{Prop.Grothendieckring3} $\chi(X_1) = \chi(X_2)$.
	\end{enumerate}	
\end{proposition}

\begin{proof}
	The implications ``\ref{Prop.Grothendieckring1} $\Rightarrow$ \ref{Prop.Grothendieckring2}'' and ``\ref{Prop.Grothendieckring2} $\Rightarrow$ \ref{Prop.Grothendieckring3}'' both follow
	from Lemma~\ref{lemm.Grothendieckring}. Proposition~\ref{prop:NormalTrees} says 
	that $X_i$ is piecewise isomorphic to the disjoint union of $\AA^2$, one point and $s_i \geq 0$ copies of $\AA^1$.
	If $\chi(X_1) = \chi(X_2)$, then we get $s_1 = s_2$, i.e.~$X_1$, $X_2$ are piecewise isomorphic.
	This gives the implication ``\ref{Prop.Grothendieckring3} $\Rightarrow$ \ref{Prop.Grothendieckring1}''.
\end{proof}

\begin{corollary}
	\label{Cor.Picewise_isom}
	Let $f, g \in \k[x, y, z, w]$ be irreducible homogeneous polynomials such that $\PP^3_f \simeq \PP^3_g$. 
	Assume moreover that the zero loci $X \coloneqq V_{\PP^3}(f)$, $Y \coloneqq V_{\PP^3}(g)$
	\begin{enumerate}[leftmargin=*, label=\alph*)]
			\item \label{Cor.Picewise_isom.a} are both normal, rational and 
			each admits a desingularisation by trees of smooth rational curves or
			\item \label{Cor.Picewise_isom.b} are both of degree $3$ and rational.
	\end{enumerate}
	Then $X$, $Y$ are piecewise isomorphic.
\end{corollary}

\begin{proof}
	Since $\PP^3_f \simeq \PP^3_g$ we get $[X] = [Y]$ inside $K_0(\Var_\k)$
	and thus $\chi(X) = \chi(Y)$, see Lemma~\ref{lemm.Grothendieckring}. In case~\ref{Cor.Picewise_isom.a},
	the statement follows from Proposition~\ref{Prop.Grothendieckring}.
	So assume we are in case~\ref{Cor.Picewise_isom.b}. Then the statement follows from
	Proposition~\ref{prop:NormalTrees} in combination with Lemma~\ref{Lem.Normal_cubic_surf}
	if $X$ and $Y$ are normal and from
	Proposition~\ref{Prop.Classification_of_non-normal_cubic_surfaces} in combination with
	Lemma~\ref{Lem.Non-Normal_cubic_surf} in the other case.
\end{proof}

We now focus on cones over smooth cubic curves, i.e.~on non-rational cubic hypersurfaces of $\PP^3$. This consists the last unproven part of Theorem~\ref{Thm.Degree3n3}. To prove that for such hypersurfaces, Question~\ref{Quest2} has an affirmative answer, we will need the following result:

\begin{lemma}\label{Lem:Gafminus1}
Let $f\in \k[x,y,z]$ be an irreducible homogeneous polynomial of degree $3$ such that $V_{\PP^2}(f)$ is a smooth cubic curve. Then, the surface 
\[X=V_{\A^3}(f-1)\]
admits no non-trivial $\GG_a$-action.
\end{lemma}
\begin{proof}
Suppose for contradiction that $X$ admits a non-trivial $\GG_a$-action. 
By \cite[Proposition 2.5.1]{BlFaTe2023Connected-Algebrai}, 
there exists a $\GG_a$-invariant affine dense open subset $X'\subseteq X$, that is a $\GG_a$-cylinder, i.e.~$X'$ is $\GG_a$-isomorphic to $\GG_a\times U$, where $U$ is a smooth affine curve and where the $\GG_a$-action
on $\GG_a \times U$ is given $s \cdot (t, u) \coloneqq (s+t, u)$ for $s, t \in \GG_a$, $u \in U$.

Using the canonical embedding $\A^3\hookrightarrow \p^3, (x,y,z)\mapsto [1:x:y:z]$ we
	can view $X$ as an open subset of the irreducible surface $Y=V_{\p^3}(f-w^3)$, 
	where $[w:x:y:z]$ denote the homogeneous coordinates on $\PP^3$.  
	Here, $X=Y_w$ is the complement  in $Y$ of a smooth curve $\Gamma$ 
	and $\Gamma \simeq V_{\PP^2}(f)$.  

Suppose first that $Y$ is a cone over a smooth cubic curve $C$. Writing $p\in Y$ the singular point, the projection away from $p$ gives to $Y\setminus \{p\}$ an $\A^1$-bundle structure over $C$. Hence, every closed rational curve on $Y$ is one of the lines through $p$. As $\Gamma$ is a hyperplane section of $Y$ and as $\Gamma$ is not rational, 
it intersects a general such line into at least one point outside of $p$. The open subset $X'\subseteq Y\setminus \Gamma$ being contained in the smooth locus of $Y$, 
it contains only finitely many curves isomorphic to $\A^1$. This contradicts the fact that $X'$ is a $\GG_a$-cylinder.

In the remaining case, $Y$ is a rational cubic surface (Lemma~\ref{Lem.Non-rational_iff_cone_over_elliptic_curve}). Hence, $U$ is a smooth affine rational curve, and is thus isomorphic to $\A^1\setminus \Delta$ for some finite set $\Delta$ of $r\ge 0$ points. Hence,  $X'$ is isomorphic to the open subset 
$V'=\A^1\times (\A^1\setminus \Delta)\subseteq \A^2$. By Corollary~\ref{Coro:ChiRatCubic}, $Y$ is piecewise isomorphic to  the disjoint union of $\A^2$, one point, and $n$ copies of $\A^1$, with 
$0 \leq n \leq 7$, and thus to the disjoint union of $X' \simeq V'$ 
with one point and $n+r$ copies of $\A^1$. As the 
one-dimensional variety $Y \setminus X'$  contains $\Gamma$, that is not rational, 
we get a contradiction, by Lemma~\ref{Lem.Example_piecewise_isom}.
\end{proof}

We are now able to prove part~\ref{Thm.Degree3n3.3} of Theorem~\ref{Thm.Degree3n3}, that is Proposition~\ref{Prop.Cones_over_elliptic_curves}. If $\k = \CC$, then Proposition~\ref{Prop.Cones_over_elliptic_curves} follows essentially from \cite{LaSe2012Birational-self-ma}.
\begin{proposition}
	\label{Prop.Cones_over_elliptic_curves}
	Let $f,g \in \kk[x, y, z, w]$ be irreducible homogeneous polynomials of degree three such that 
	$V_{\PP^3}(f)$, $V_{\PP^3}(g)$ are non-rational.
	If $\PP^3_f \simeq \PP^3_g$, then there exists $\varphi \in \Aut(\PP^3)$ with
	$\varphi(V_{\PP^3}(f)) = V_{\PP^3}(g)$.
\end{proposition}
\begin{proof}	By Lemma~\ref{Lem.Non-rational_iff_cone_over_elliptic_curve} we may assume that
	$f,g \in \kk[x, y, z]$ and that the zero loci $V_{\PP^2}(f)$, $V_{\PP^2}(g)$ are smooth cubic curves.
	
	By Lemma~\ref{Lem:Gafminus1}, every $\GG_a$-action on the affine surfaces $X_0 \coloneqq V_{\AA^3}(f-1)$ and $Y_0 \coloneqq V_{\AA^3}(g-1)$ is trivial.
	
	Let $\theta \colon \PP_f^3 \to \PP^3_g$ be an isomorphism. By Proposition~\ref{Prop.cyclic_covering}
	and Remark~\ref{Rem.mu_equal_1}, 
	we get an isomorphism 
	\[
		\varphi \colon \AA^1 \times X_0 = V_{\AA^4}(f-1)  \to V_{\AA^4}(g-1) = \AA^1 \times Y_0
	\]
	such that $\pi_{g, 1} \circ \varphi =  \theta  \circ \pi_{f, 1}$, where 
	the morphisms $\pi_{f, 1} \colon V_{\AA^4}(f-1) \to \PP^3_f$ and
	$\pi_{g, 1} \colon V_{\AA^4}(g-1) \to \PP^3_g$ denote the canonical projections. 
	
	Using that $X_0$ admit no non-trivial $\GG_a$-action,  by \cite[Proposition~4.7]{Cr2004On-the-AK-invarian} it follows that 
	the intersection of the invariant subalgebras $\kk[\AA^1 \times X_0]^{\GG_a}$ inside $\kk[\AA^1 \times X_0]$ over all
	$\GG_a$-actions on $\AA^1 \times X_0$ is equal to $\kk[X_0]$.
	Using that a similar statement holds for $Y_0 \times \AA^1$, we get that $\varphi$
	maps the fibres of $\AA^1 \times X_0 \to X_0$ onto the fibres of $\AA^1 \times Y_0 \to Y_0$.
	Using the commutative diagrams
	\[
	\xymatrix@=20pt{
		\AA^1 \times X_0 \ar[d]_-{(t, x_0) \mapsto x_0} \ar[r]^-{\pi_{f, 1}} &  \PP^3_f 
		\ar[d]^{ \substack{ \textrm{projection} \\ \textrm{from $[1:0:0:0]$} } } \\
		X_0 \ar[r] &  \PP^2_f
	}
	\quad
	\xymatrix@=20pt{
		\AA^1 \times Y_0 \ar[d]_-{(t, y_0) \mapsto y_0} \ar[r]^-{\pi_{g, 1}} &  \PP^3_g 
		\ar[d]^{ \substack{ \textrm{projection} \\ \textrm{from $[1:0:0:0]$} } } \\
		Y_0 \ar[r] &  \PP^2_g
	}
	\]
	we get that $\theta \colon  \PP_f^3 \to \PP^3_g$ maps the fibres of 
	the locally trivial $\AA^1$-bundle $\PP^3_f \to \PP^2_f$ (with respect to the Zariski topology) 
	onto those of $\PP^3_g \to \PP^2_g$. Hence, we get an isomorphism $\PP^2_f \simeq \PP^2_g$.
	As $V_{\PP^2}(f)$ and $V_{\PP^2}(g)$ are non-rational this isomorphism extends to an automorphism of $\PP^2$.
\end{proof}

Now we are able to give the proof of Theorem~\ref{Thm.Degree3n3}:

\begin{proof}[Proof of Theorem~\ref{Thm.Degree3n3}]
	\ref{Thm.Degree3n3.1}: This follows from Corollary~\ref{Cor.Picewise_isom}.
	
	\ref{Thm.Degree3n3.3}: If $H$ is a non-rational cubic surface, then as before, $H'$ is a cubic as well.
	Since $\chi(H) = \chi(H')$, Lemma~\ref{Lem.cubic_Euler_char} implies that $H'$ is non-rational. 
	Thus~\ref{Thm.Degree3n3.3} follows from Proposition~\ref{Prop.Cones_over_elliptic_curves}.
	
	\ref{Thm.Degree3n3.2}: If $H$ or $H'$ is a cubic surface, then both are cubics (by using the order of the Picard groups 
	of the complements). If either $H$ or $H'$ is non-rational, then \ref{Thm.Degree3n3.3} implies that $H$ and $H'$ are isomorphic, and thus also piecewise isomorphic. We may thus assume that both $H$ and $H'$ are rational. By Corollary~\ref{Coro:ChiRatCubic}
	it follows that $H$ (respectively $H'$) is piecewise isomorphic to a disjoint union of $\AA^2$, one point and
	$n$ (respectively $n'$) copies of $\AA^1$. As $\PP^3 \setminus H$ and $\PP^3 \setminus H'$
	are isomorphic, we get $2 + n  = \chi(H) = \chi(H') = 2 + n'$ and hence $H$, $H'$ are piecewise isomorphic.
	If moreover $H$ is smooth, then $\chi(H) = 9$ by Lemma~\ref{Lem.Normal_cubic_surf}.
	If $H'$ would be non-normal, then Lemma~\ref{Lem.Non-Normal_cubic_surf} and Proposition~\ref{Prop.Classification_of_non-normal_cubic_surfaces} would imply that $\chi(H') \leq 4$,
	which contradicts $\chi(H) = \chi(H')$. Thus Lemma~\ref{Lem.Normal_cubic_surf} implies that $H'$
	is smooth. This finishes the proof of~\ref{Thm.Degree3n3.2}.
\end{proof}

We finish this subsection  with the following concrete question motivated by Proposition~\ref{Prop.Counterexample.dim4}:

\begin{question}
	Let $f, g \in \k[x, y, z, w]$ be given by
	\[
		f = x^2 y + z^3 \quad \textrm{and} \quad g = x^2 y + z^3 + x w^2 \, .
	\]
	Are the varieties $\PP^3_f$ and $\PP^3_g$ isomorphic?
\end{question}

An affirmative answer would give a negative answer to Question~\ref{Quest3} 
in degree three and dimension three,
since the zero locus of $f$ in $\PP^3$ is non-normal, whereas the zero locus of $g$ in $\PP^3$ is normal.
Moreover, $\chi(V_{\PP^3}(f)) = \chi(V_{\PP^3}(g))$, since $\PP^4_f \simeq \PP^4_g$ by Proposition~\ref{Prop.Counterexample.dim4} and since
$\PP^4_f \to \PP^3_f$, $\PP^4_g \to \PP^3_g$ are locally trivial $\AA^1$-bundles with respect to the Zariski topology,
see~Example~\ref{Exa.A1bundle}.

\subsection{Automorphisms of complements of singular cubic surfaces}
Throughout this subsection we always assume that $\kk$ is algebraically closed.

We finish this text by giving some examples, that give a negative answer to Question~\ref{Quest1}, for singular irreducible cubic hypersurfaces of $\p^3$. As explained in the introduction, this question is wide open for smooth cubics and known to have a negative answer for singular cubic surfaces with Du Val singularities
in characteristic zero \cite[Theorem C and Theorem~4.3]{CheDub}. We now extend this to other singular cubics.
Each irreducible cubic hypersurface of $\p^3$ that does not have Du Val singularities  is either the cone
	over a smooth cubic curve in $\PP^2$ or is rational and non-normal (this follows from Lemmas~\ref{Lem.Non-rational_iff_cone_over_elliptic_curve} and~\ref{Lem.Normal_cubic_surf}). In this latter case, it always contains a line with multiplicity $2$ (Proposition~\ref{Prop.Classification_of_non-normal_cubic_surfaces}). These two cases are done in the next two simple lemmas, that work in any dimension.

\begin{lemma}\label{Lem:Cone}
Let $d\ge 1$, $n\ge 3$ and let $X\subseteq \p^n$ be an irreducible hypersurface of degree $d$, having a point of multiplicity $d$ $($i.e.~being a cone$)$. Then, there exists an element $\varphi\in \Aut(\p^n\setminus X)$ that does not extend to an element of $\Aut(\p^n)$.
\end{lemma}
\begin{proof}
Changing coordinates, we may assume that $[0:\cdots:0:1]$ is a point of $X$ of multiplicity $d$. Hence, $X=V_{\p^n}(f)$, where $f\in \k[x_0,\ldots,x_{n-1}]$ is irreducible and homogeneous of degree $d$. We define $\varphi\in \Aut(\p^n_f)$ to be the involution given by 
\[ 
	[x_0:\cdots:x_n]\mapsto \left[x_0:\cdots:x_{n-1}:-x_n+ \frac{x_0^{d+1}}{f(x_0,\ldots,x_{n-1})} \right] \, .
\]
\end{proof}

\begin{lemma}\label{Lem:LinSubdm1}
Let $d\ge 1$, $n\ge 3$ and let $X\subseteq \p^n$ be an irreducible hypersurface of degree $d$, being of multiplicity $d-1$ along a linear subspace $L\subseteq \p^n$ of dimension $r\in \{1,\ldots,n-2\}$. Then, there exists an element $\varphi\in \Aut(\p^n\setminus X)$ that does not extend to an element of $\Aut(\p^n)$.
\end{lemma}
\begin{proof}
Changing coordinates, we may assume that $L=V_{\p^n}(x_{r+1},\ldots,x_n)$. Hence, $X=V_{\p^n}(f)$, where $f\in \k[x_0,\ldots,x_{n}]$ is irreducible, homogeneous of degree $d$ and of the form
\[
	f=x_0a_0+\cdots+x_ra_r+b \, , 
\]
where $a_0,\ldots,a_r,b\in \k[x_{r+1},\ldots,x_n]$. We define a $\mathbb{G}_a$-action on $\p^n_f$ by
\[
	\begin{array}{ccc}
		\mathbb{G}_a\times \p^n_f& \to &\p^n_f\\
		(t, [x_0:\cdots:x_n]) &\mapsto& \left[x_0+t \frac{a_1 x_n^2}{f(x_0,\ldots,x_n)}:x_1-t \frac{a_0 x_n^2}{f(x_0,\ldots,x_n)}:x_2:\cdots:x_n\right]
	\end{array}
\]
and choose $\varphi\in \Aut(\p^n_f)$ to be any non-trivial element of $\mathbb{G}_a\subseteq  \Aut(\p^n_f)$.
\end{proof}

The following lemma follows from \cite[Theorem C and Theorem~4.3]{CheDub} in characteristic zero.
We insert a proof (in any characteristic), as the argument is simple.

\begin{lemma}
	\label{Lem.Contains_Cylinder}
	Let $X \subseteq \PP^3$ be a normal rational singular cubic hypersurface. Then 
	$\PP^3 \setminus X$ contains an open affine subset $U$ that is an $\AA^1$-cylinder, i.e.~$U \simeq \AA^1 \times U'$
	for some affine variety $U'$.
\end{lemma}

\begin{proof}
	Since $X$ is singular, we may assume (after a coordinate change) that $X$ is given by 
	$f \coloneqq f_2(x, y, z) w + f_3(x, y, z)$
	where $f_i \in \kk[x, y, z]$ is homogeneous of degree $i$ and $[w:x:y:z]$ denote the homogeneous coordinates of $\PP^3$. As $X$ is normal and rational, $f_2 \neq 0$, see Lemma~\ref{Lem.Non-rational_iff_cone_over_elliptic_curve}. 
	Let $C \coloneqq V_{\PP^2}(f_2)$. We may choose the coordinates $(x,y, z)$ 
	in such a way, that $V_{\PP^2}(z)$ is contained in $C$ if $f_2$ is reducible, and
	that $V_{\PP^2}(z)$ is tangent to $C$ if $f_2$ is irreducible. Then
	$\p^2_{zf_2}\simeq \AA^2_{f_2(x, y, 1)} \simeq D \times \AA^1$, where $D \in \{\AA^1, \AA^1 \setminus \{0\} \}$, 
	and we get an isomorphism
	\[
		\begin{array}{cccc}
			\PP^3_{zff_2} \simeq &\AA^3_{f(w, x, y, 1) f_2(x, y, 1)} &  \xrightarrow{\simeq} &
			(\AA^1 \setminus \{0\}) \times \AA^2_{f_2(x, y, 1)} \\
			&(w, x, y) & \mapsto & (f_2(x, y, 1)w + f_3(x, y, 1), x, y) \\
				&\left(\frac{u - f_3(x, y, 1)}{f_2(x, y, 1)} , x, y \right) 
						& \mapsfrom & (u, x, y) \, .
		\end{array}
	\]
\end{proof}

\begin{proposition}\label{Prop:SingCubic}
For each singular irreducible cubic hypersurface $X\subseteq \p^3$, there exists an element $\varphi\in \Aut(\p^3\setminus X)$ that does not extend to an element of $\Aut(\p^3)$.
\end{proposition}
\begin{proof}
If $X$ is a cone, this follows from Lemma~\ref{Lem:Cone}. If $X$ is not normal, it contains a line with multiplicity $2$ (Proposition~\ref{Prop.Classification_of_non-normal_cubic_surfaces}), and the result follows from Lemma~\ref{Lem:LinSubdm1}. In the remaining case, $X$ is rational (Lemma~\ref{Lem.Non-rational_iff_cone_over_elliptic_curve}) and normal. 
As $\Pic(\PP^3 \setminus X) \simeq \ZZ/3\ZZ$ is finite, 
Lemma~\ref{Lem.Contains_Cylinder}
implies that $\PP^3 \setminus X$ admits a non-trivial $\GG_a$-action, see e.g.~\cite[Propostion~2]{DuKi2015Log-uniruled-affin}. This implies the result.
\end{proof}

\bibliographystyle{alpha}
\bibliography{BIB}

\end{document}